\documentclass[12pt]{article}
\usepackage{amssymb}
\usepackage{amsmath}
\usepackage{amsfonts,color}

\newtheorem{theorem}{Theorem}

\newcounter{cor}
\newtheorem{corollary}[cor]{Corollary}

\newtheorem{example}[theorem]{Example}

\newcounter{lem}
\newtheorem{lemma}[lem]{Lemma}
\newcounter{prop}

\newcounter{rem}
\newtheorem{remark}[rem]{Remark}

\newenvironment{proof}[1][Proof]{\noindent\textbf{#1.} }{\ \rule{0.5em}{0.5em}}

\begin{document}

\title{Semiparametric efficiency bounds for seemingly unrelated conditional
moment restrictions}
\author{Marian Hristache \& Valentin Patilea \\
CREST (Ensai)\footnote{CREST, Ecole Nationale de la Statistique et de l'Analyse
 de l'Infdormation (Ensai), Campus de Ker-Lann, rue Blaise Pascal, BP 37203,
 35172 Bruz, cedex, France. Authors emails: hristach@ensai.fr, patilea@ensai.fr }}
\maketitle

\begin{abstract}
This paper addresses the problem of semiparametric efficiency bounds for
conditional moment restriction models with different conditioning variables. We
characterize such an efficiency bound, that in general is not explicit, as a
limit of explicit efficiency bounds  for a decreasing sequence of unconditional
(marginal) moment restriction models. An iterative procedure for approximating
the efficient score when this is not explicit is provided. Our theoretical
results complete and extend existing results in the literature, provide new
insight for the theory of semiparametric efficiency bounds literature and
open the door to new applications. In particular, we investigate a class of
regression-like (mean regression, quantile regression,...) models with missing data.
\end{abstract}

\setcounter{theorem}{0} \setcounter{lem}{0} \setcounter{cor}{0} %
\setcounter{prop}{0}

\section{The model}

Conditional moment restriction models represent a large class of statistical models.
Seemingly unrelated nonlinear regressions, see Gallant (1975), M\"{u}ller (2009),
seemingly unrelated quantile  regressions, see Jun and Pinske (2009), regression
models with missing data, see Robins, Rotnitzky and Zhao~(1994),  Tsiatis~(2006),
are only few examples and related contributions. Ai and Chen (2009) and Hansen (2007)
provide many other references and examples of  econometric models that could be
stated  as conditional moment restriction models.

In this paper we address the problem of calculating semiparametric efficiency
bounds in models defined by several conditional moment restrictions with possibly
different conditioning variables.
More formally, the sample under study consists of independent copies of a random vector $%
Z\in \mathcal{Z}\subset \mathbb{R}^{q}$. Let $J$ be some positive integer
that is fixed in the following. For any $j\in \left\{ 1,\ldots ,J\right\} $,
let $X^{\left( j\right) }$ be a random $q_{j}-$dimension subvector of $Z$,
where $0\leq q_{j}<q$. Let $g_{j}:\mathbb{R}^{q}\mathbb{\times R}%
^{d}\rightarrow \mathbb{R}^{p_{j}}$, $j\in \left\{ 1,\ldots ,J\right\} ,$
denote given functions of $Z$ and the unknown parameter $\theta \in \Theta
\subset \mathbb{R}^{d}$. The semiparametric model we consider is defined by
the conditional moment restrictions
\begin{equation}
E\left[ g_{j}\left( Z,\theta \right) \ |\ X^{\left( j\right) }\right]
=0,\quad j=1,\ldots ,J,\ \ \ \text{almost surely}.  \label{model_I}
\end{equation}%
It is assumed that the $d-$dimension parameter $\theta $ is
identified by the conditional restrictions, which means  there exists a unique
value $\theta _{0}$ such that the true law of $Z$ satisfies
equations (\ref{model_I}). By definition, $X^{(j)}$ is a constant random variable when $q_{j}=0$, and
hence the conditional expectation given $X^{(j)}$ is the marginal
expectation.

Particular cases of this model have been extensively studied in the
literature. For $J=1$ and  $q_{1}=0$ we obtain a model defined by an
unconditional set of moment equations%
\begin{equation*}
E\left[ g\left( Z,\theta \right) \right] =0.
\end{equation*}%
Hansen~(1982) considered the class
of GMM estimators and showed how to construct an optimal one in this class.
Its asymptotic variance equals the the semiparametric efficiency bound
obtained by Chamberlain~(1987).

The GMM method extends naturally to models
defined by conditional moment equations, corresponding to the case $J=1$
and $q_{1}>0$ in our setting, that is
\begin{equation*}
E\left[ g\left( Z,\theta \right) \ |\ X\right] =0.
\end{equation*}%
From a mathematical point of view, such a model is equivalent to the
intersection of the models of the form%
\begin{equation*}
E\left[ a\left( X\right) \ g\left( Z,\theta \right) \right] =0,
\end{equation*}%
where $a\left( X\right) $ is an arbitrary conformable random matrix
whose entries are square integrable. Following the econometric literature, $a\left( X\right) $ is referred to as a matrix of \emph{instruments}.
The supremum of the information on $%
\theta _{0}$ in these models yields the semiparametric Fisher information on
$\theta _{0}$ in the conditional equation model, obtained by Chamberlain
(1992a). It is also the information on $\theta _{0}$ for the unconditional
moment equation%
\begin{equation*}
E\left[ a^{\ast }\left( X\right) \ g\left( Z,\theta \right) \right] =0,
\end{equation*}%
with properly chosen `optimal' instruments $a^{\ast }\left( X\right) $.

A
further generalization, which can also be written under the form $(\ref{model_I}) $,
is given by a sequential (nested) moment restrictions model, in
which the $\sigma-$fields generated by the conditioning vectors satisfy
the condition $\sigma \left( X^{\left( 1\right) }\right) \subset \sigma \left(
X^{\left( 2\right) }\right) \subset \ldots \subset \sigma \left( X^{\left(
J\right) }\right) $. For the expression of the semiparametric efficiency
bound in the sequential case,  see Chamberlain (1992b) and Ai and
Chen~(2009); see also Hahn~(1997) and Ahn and Schmidt~(1999) and references
therein for examples of applications. It turns out that once again
the information on $\theta _{0}$ can be obtained by taking the supremum of
the information on $\theta _{0}$ in the following unconditional models~:%
\begin{equation*}
E\left[ a_{j}\left( X^{\left( j\right) }\right) \ g_{j}\left( Z,\theta
\right) \right] =0,\quad j=1,\ldots ,J,
\end{equation*}%
where the number of lines of the matrices $a_{j}$ is fixed and equal to the dimension of $\theta$
and the supremum is attained for a suitable choice $a_{1}^{\ast }\left(
X^{\left( 1\right) }\right) ,\ldots ,a_{J}^{\ast }\left( X^{\left( J\right)
}\right) $ of optimal instruments. The reason why this happens
in the case with nested $\sigma-$fields is the fact that the model of interest can be written as the decreasing limit of a sequence of models for which a so-called \emph{`spanning condition'}, similar to
the one considered in Newey~(2004), holds and the limit of the corresponding
efficient scores has an explicit solution.

In this paper we  show that the information on $\theta _{0}$ in model $%
\left( \ref{model_I}\right) $ can be obtained as the limit of the
information on $\theta _{0}$ in a decreasing sequence of unconditional
moment models of the form%
\begin{equation}\label{ezr}
E\left[ a_{j}^{\left( k\right) }\left( X^{\left( j\right) }\right) \
g_{j}\left( Z,\theta \right) \right] =0,\quad j=1,\ldots ,J,\quad k=1,2,\cdots
\end{equation}%
where the numbers of lines in the matrices $a_{j}^{\left( k\right)
}$ increases to infinity with $k$. To our best knowledge this result
is new. It provides theoretical support for a natural solution that
could be used in practice: replace the model (\ref{model_I}) by a
large number of unconditional moment conditions like (\ref{ezr}) in
order to approach efficiency. Herein we also propose an alternative
route for approximating the efficiency bound. More precisely,  we
give a general method to approximate the efficient score, which in
most of the situations does not have an explicit form as in the
aforementioned examples. In particular, our general approach for
approximating the efficient score brings in a new light the
functional equations used to characterize the efficient score in the
regression model with unobserved explanatory variables in Robins,
Rotnitzky and Zhao~(1994); see also Tsiatis~(2006) and Tan~(2011).
To summarize, our theoretical results complete and extend existing
results in the literature, provide new insight for the theory of
semiparametric efficiency bounds literature and open the door to new
applications, in particular in missing data contexts.

The paper is organized as follows. Section \ref{main_result_sec} contains our main results. We show that under a suitable  `spanning condition' on the tangent spaces, the  semiparametric Fisher
information in model (\ref{model_I}) can be obtained
as the limit of the efficiency bounds for a decreasing sequence of models.
In section \ref{Section : Efficient estimation} we propose a `backfitting' procedure, for computing the projection of the score on the tangent space of the model. With at hand an approximation of the efficient score, we suggest a general method for constructing asymptotically efficient estimators. In section \ref{sec_applis_b} we illustrate we illustrate the utility of our theoretical results for two large classes of
models: sequential (nested) conditional models and regression-like models with missing data.
The technical assumptions required for our results and some technical proofs are relegated to the Appendix.

\section{The main results}

\label{main_result_sec} Let us introduce some notation and definitions, see
also van der Vaart (1998), sections 25.2 and 25.3. Given a sample space $%
\mathcal{Z}$ and a probability $P$ on the sample space, we denote by $L^2(P)$
the usual Hilbert space of measurable real-valued functions that are
squared-integrable with respect to $P.$ For $\mathcal{H}$ a Hilbert space
and $\mathcal{S}\subset \mathcal{H}$ let $\overline{S}$ denote the closure
of $\mathcal{S}$ in $\mathcal{H}$. Moreover, if $\mathcal{S}\subset \mathcal{%
H}$ is a linear subspace and $h\in \mathcal{H}$, let $\Pi(h | \mathcal{S})$
be the projection of $h$ on $\overline{S}$. The statistical models on the
sample space $\mathcal{Z}$, are denoted by $\mathcal{P}$, $\mathcal{P}_{1}$,
$\mathcal{P}_{2}$... A statistical model is a collection of probability
measures defined by their densities with respect to some fixed dominating
measure on the sample space. For a model $\mathcal{P}$ (resp. $\mathcal{P}%
_{j}$) and a probability measure $P$ in the model, let $\mathcal{\dot{P}}%
_{P} $ (resp. $\mathcal{\dot{P}}_{j,P}$) denote the tangent cone of the
model $\mathcal{P}$ (resp. $\mathcal{P}_{j}$) at $P$. When there is no
possible confusion, we simply write $\mathcal{\dot{P}}_{P}$ (resp. $\mathcal{%
\dot{P}}_{j,P}$). Let $\mathcal{T}( \mathcal{P},P )$ denote the tangent
space of a model $\mathcal{P}$ at some probability measure $P\in\mathcal{P}$%
, that means the closure of the linear span of the tangent set $\mathcal{%
\dot{P}}_{P}$. By definition, both the tangent cone and the tangent space
are subsets of $L^2(P)$. Herein the vectors are column matrices and $A \in
\mathbb{R}^r \times \mathbb{R}^s$ means $A$ is a $r\times s-$matrix with
random elements, if not stated differently. For $A\in \mathbb{R}^r \times
\mathbb{R}^r,$ $E(A)$ denotes the expectation of $A$ and $E^{-1}(A)$ denotes
the inverse of the square matrix $E(A)$. Finally, for a square matrix $A$,
let $A^-$ denote a generalized inverse, for instance the Moore-Penrose
pseudoinverse.

\subsection{A general lemma}

The following result is a generalization of Theorem 1 in Newey
(2004) where only the case of conditioning vectors $X^{(j)}$,
$j=1,\cdots,J,$ that generate the same $\sigma-$field is considered.
The proof of our result is postponed to the Appendix.

\begin{lemma}
\label{key_lemma} Let $P_{0}\in \mathcal{P\subset P}_{1}$ be the true law of
the vector $Z\in\mathcal{Z}$ and $\theta_0 = \psi(P_0)$ for a map $\psi :%
\mathcal{P}_{1}\rightarrow \mathbb{R}^{d}$ differentiable at $P_{0}$
relative to the tangent cone $\mathcal{\dot{P}}_{1,P_{0}}$. Let $\left\{
\mathcal{P}_{k}\right\} _{k\in \mathbb{N}^{\ast }}$ be a decreasing family
of statistical models such that%
\begin{equation}
\mathcal{P}_{1}\supset \mathcal{P}_{2}\supset \ldots \supset \mathcal{P}%
_{j}\supset \mathcal{P}_{k+1}\supset \ldots \supset
\bigcap\limits_{k=1}^{\infty }\mathcal{P}_{k}\supset \mathcal{P}\ni P_{0}
\label{decreasing}
\end{equation}%
and
\begin{equation}
\bigcap\limits_{k=1}^{\infty }\mathcal{T}_{k}=\mathcal{T},
\label{spaning cond 2}
\end{equation}%
where $\mathcal{T}=\mathcal{T}\left( \mathcal{P},P_{0}\right) $ and $%
\mathcal{T}_{k}=\mathcal{T}\left( \mathcal{P}_{k},P_{0}\right) ,$ $k\in
\mathbb{N}^{\ast }$. Then%
\begin{equation*}
I_{\theta _{0}}\left( \mathcal{P}\right) =\lim_{k\rightarrow \infty
}I_{\theta _{0}}\left( \mathcal{P}_{k}\right) ,
\end{equation*}%
where $I_{\theta _{0}}\left( \mathcal{P}\right) $ stands for the Fisher
information on $\theta _{0}=\psi \left( P_{0}\right) $ in the model $%
\mathcal{P}$.
\end{lemma}

\medskip

For the definition of the Fisher information $I_{\theta _{0}}\left( \mathcal{%
P}\right) $ on $\theta _{0}=\psi \left( P_{0}\right) $ in the model $%
\mathcal{P}$ we refer to Bickel, Klaassen, Ritov and Wellner (1993) or van
der Vaart (1998); see also Newey (1990). When the models $\mathcal{P}_{k}$, $%
k\in \mathbb{N}^{\ast }$, are defined by an increasing number of moment
conditions with the same conditioning vectors, condition $\left( \ref%
{spaning cond 2}\right) $ is exactly the so-called spanning condition of
Newey~(2004).

\medskip

\begin{remark}
\label{remarca_isteata} Even if $\bigcap_{k=1}^{\infty }\mathcal{P}_{k}=%
\mathcal{P}$, condition $\left( \ref{spaning cond 2}\right) $ is not
necessarily fulfilled. To see this, consider a symmetric density $f_{0}$ on
the real line and let $s_1$, $s_2$ be two odd functions such that $%
|s_1|,|s_2|\leq 1$ (e.g. $s_{l}( x) =x^{2l-1} \mathbf{I}_{\{|x|\leq 1\}}$, $%
l=1,2$). For any $k\in\mathbb{N}^{\ast }$ and $t\in [ -1,1]$, define
\begin{equation*}
f_{t}( x) = f_{0}( x)[ 1+t\ s_{2}( x) ]
 ,\qquad
f_{t;k}( x) =k f_{0}( k x) [ 1+t\ s_{1}( x)] 
\end{equation*}
and consider the following models defined by theirs densities with respect
to $\lambda _{\mathbb{R}}$ the Lebesgue measure on the real line~: $\mathcal{%
Q}_{k} =\left\{ f_{t;k}\cdot \lambda _{\mathbb{R}}: t\in \left[ -1,1\right]
\right\}$, $k\in \mathbb{N}^{\ast },$ and
\begin{equation*}
\mathcal{P} =\left\{ f_{t}\cdot \lambda _{\mathbb{R}}: t\in \left[ -1,1%
\right] \right\} , \ \qquad \mathcal{P}_{k} =\mathcal{P}\cup
\bigcup\limits_{m=k}^{\infty }\mathcal{Q}_{m} ,\quad k\in \mathbb{N}^{\ast }.
\end{equation*}
Then we have%
\begin{equation*}
\mathcal{P}_{1}\supset \mathcal{P}_{2}\supset \ldots \supset \mathcal{P}%
_{k}\supset \mathcal{P}_{k+1}\supset \ldots \supset
\bigcap\limits_{k=1}^{\infty }\mathcal{P}_{k}=\mathcal{P} .
\end{equation*}%
To describe the corresponding tangent spaces, notice that%
\begin{equation*}
\forall k\geq 1,\;\;\partial _{t}\left. \log f_{t;k}\left( x\right)
\right\vert _{t=0}=s_{1}\left( x\right) \quad \text{and}\quad \partial
_{t}\left. \log f_{t}\left( x\right) \right\vert _{t=0}=s_{2}\left( x\right)
,
\end{equation*}%
and thus $\mathcal{\dot{P}} =\left\{ a\, s_{2}\left( x\right) : a\in \mathbb{%
R}\right\} ,$
\begin{equation*}
\mathcal{\dot{P}}_{k} =\left\{ a\, s_{2}\left( x\right) : a\in \mathbb{R}%
\right\} \cup \left\{ b\, s_{1}\left( x\right) : b\in \mathbb{R}\right\}
,\quad k\in \mathbb{N}^{\ast }.
\end{equation*}
Then $\mathcal{T} =\left\{ a\, s_{2}\left( x\right) : a\in \mathbb{R}%
\right\} ,$
\begin{equation*}
\mathcal{T}_{k} =\left\{ a\, s_{2}\left( x\right) +b\, s_{1}\left( x\right)
: a,b\in \mathbb{R}\right\} ,\quad k\in \mathbb{N}^{\ast } .
\end{equation*}
This shows that
\begin{equation*}
\bigcap\limits_{k=1}^{\infty }\mathcal{\dot{P}}_{k}\varsupsetneq \mathcal{%
\dot{P}}\qquad \text{and} \qquad \bigcap\limits_{k=1}^{\infty }\mathcal{T}%
_{k}\varsupsetneq \mathcal{T},
\end{equation*}
even if the decreasing sequence of models $\left\{ \mathcal{P}_{k}\right\}
_{k\in \mathbb{N}^{\ast }}$ is such that $\bigcap_{k=1}^{\infty }\mathcal{P}%
_{k}=\mathcal{P}. $
\end{remark}

\subsection{Efficiency bound}

The main idea we follow to derive the semiparametric efficiency bound for
the parameter $\theta_0$ is to transform the finite number of conditional
moment restrictions (\ref{model_I}) in a countable number of unconditional
(marginal) moment restrictions. Next, for any finite subset of these
unconditional moment restrictions, one could easily obtain the Fisher
information bound. Eventually, one may expect to obtain the semiparametric
efficiency bound for the model (\ref{model_I}) as the limit of the
efficiency bounds for a decreasing sequence of models defined by an
increasing sequence of finite subsets of unconditional moment restrictions.
Remark \ref{remarca_isteata} proves that in general this intuition is not
correct. However, Lemma \ref{key_lemma} states that this intuition becomes
correct under the additional condition (\ref{spaning cond 2}).

Let us introduce some more notation. If $\zeta:\mathcal{Z}\times
\Theta\rightarrow \mathbb{R}^m$, $m\geq 1$, is some given function of $Z$
and $\theta$ and $X$ is some subvector of $Z$, we denote
\begin{equation}  \label{eq_biz}
E[\partial_{\theta ^{\prime}}\zeta \mid X] = E[\partial_{\theta
^{\prime}}\zeta(Z,\theta_0)\mid X]= \left. \frac{\partial}{\partial
\theta^\prime} E[\zeta(Z,\theta_0)\mid X] \right|_{\theta=\theta_0}\in\mathbb{%
R}^{d}\times \mathbb{R}^{m},
\end{equation}
when such derivatives of $\theta\mapsto E[\zeta(Z,\theta)\mid X]$ exist. A similar notation will be used with the
conditional expectation $E(\cdot \mid X)$ replaced by the marginal
(unconditional) expectation with respect to the law of $Z$. Let us point out
that the maps $\theta\mapsto \zeta(z,\theta)$ may not be everywhere
differentiable. Next, let us define
\begin{equation*}
\underline{g}=\left( g_{1}^\prime,\cdots, g_{J}^\prime\right)^\prime \in
\mathbb{R}^{p}=\mathbb{R}^{p_{1}+\ldots +p_{J}},
\end{equation*}
and let $\underline{X}$ denote the vector of all components of $Z$ contained
in the subvectors $X^{(j)}$, $j=1,\cdots,J.$

For the purpose of transforming conditional moments in unconditional
versions, consider a countable set of squared integrable functions $\mathcal{%
W}=\left\{ w_{k}: k\in \mathbb{N}^{\ast } \right\} \subset L^{2}\left(
P_0\right) $ such that $\overline{\mathrm{lin}} \mathcal{W} = L^{2}\left(
P_0\right)$, that is the linear span of $\mathcal{W}$ is dense in $%
L^{2}\left( P_0\right)$. For any $s\in \mathbb{N}^{\ast }$, define a $%
p\times p-$diagonal matrix
\begin{eqnarray*}
\underline{w}_{s}\left( \underline{X}\right) &=&diag(\underset{p_{1}}{%
\underbrace{E\left[ w_{s}\left( Z\right) |X^{\left( 1\right) }\right]
,\ldots ,E\left[ w_{s}\left( Z\right) |X^{\left( 1\right) }\right] }},\ldots
\\
&&\qquad \qquad \qquad \qquad \ldots ,\underset{p_{J}}{\underbrace{E\left[
w_{s}\left( Z\right) |X^{\left( J\right) }\right] ,\ldots ,E\left[
w_{s}\left( Z\right) |X^{\left( J\right) }\right] }}).
\end{eqnarray*}%
Next, for any $k\in \mathbb{N}^{\ast }$, let
\begin{equation*}
\underline{w}^{(k) }(\underline{X}) = \left( \underline{w}_{1}( \underline{X}%
), \cdots, \underline{w}_{k}(\underline{X})\right)^\prime \in \mathbb{R}^{k
p} \times \mathbb{R}^{p} \quad \text{and}\quad \underline{g}_{k}^{w}\left(
Z,\theta \right) =\underline{w}^{\left( k\right) }\left( \underline{X}%
\right) \underline{g}\left( Z,\theta \right) \in \mathbb{R}^{kp}.
\end{equation*}
Moreover, let $I_{\theta _{0}}^{\left( k\right) }$ be the Fisher information
on $\theta _{0}$ in the model%
\begin{equation}
E\left[ \underline{g}_{k}^{w}\left( Z,\theta \right) \right] =0,
\label{model_III}
\end{equation}%
that is
\begin{equation*}
I_{\theta _{0}}^{\left( k\right) }= E\left[ \left( \partial_{\theta ^{\prime
}}\underline{g}_{k}^{w}\left( Z,\theta _{0}\right) \right) ^{\prime }\right]
V^{-}\left[ \underline{g}_{k}^{w}\left( Z,\theta _{0}\right) \right] E\left[
\partial_{\theta ^{\prime }}\underline{g}_{k}^{w}\left( Z,\theta _{0}\right) %
\right] .
\end{equation*}
See Chamberlain~(1987), Newey (2001), see also Chen and Pouzo (2009) for the
non-smooth case.

We can state now the main result of the paper.

\begin{theorem}
\label{Main_Result} Under the Assumptions $T$ and $SP$ in the Appendix, the
information bound $I_{\theta _{0}}$ on $\theta _{0}$ at $P_{0}$ in model $%
\left( \ref{model_I}\right) $ is given by
\begin{equation*}
I_{\theta _{0}}=\lim_{k\rightarrow \infty }I_{\theta _{0}}^{\left( k\right)
},
\end{equation*}%
where, for any $k\in\mathbb{N}^{\ast }$, $I_{\theta _{0}}^{\left( k\right) }$
is the Fisher information on $\theta _{0}$ in the model defined as in $%
\left( \ref{model_III}\right) $.
\end{theorem}

\begin{proof}
For any $k\in \mathbb{N}^{\ast }$, let $\mathcal{P}_{k}$ be the model
defined by equation $\left( \ref{model_III}\right) $ and $\mathcal{P}$ the
model defined by equation $\left( \ref{model_I}\right). $ Then%
\begin{equation*}
\mathcal{P}_{1}\supset \mathcal{P}_{2}\supset \ldots \supset \mathcal{P}%
_{k}\supset \mathcal{P}_{k+1}\supset \ldots \supset
\bigcap\limits_{k=1}^{\infty }\mathcal{P}_{k}=\mathcal{P}.
\end{equation*}%
Hence the stated result is a direct consequence of Lemma~\ref{key_lemma},
provided that condition $\left( \ref{spaning cond 2}\right) $ holds for the
tangent spaces of $\mathcal{P}$ and $\mathcal{P}_{k}$, $k\in \mathbb{N}%
^{\ast }$, at $\theta _{0}$.

For each $j\in \left\{ 1,\ldots ,J\right\} $, any $z\in\mathcal{Z}\subset%
\mathbb{R}^q$ could be partitioned in two subvectors $y^{\left( j\right) }\in%
\mathbb{R}^{q-q_j}$ and $x^{\left( j\right) }\in\mathbb{R}^{q_j}$ with $%
x^{\left( j\right) }$ in the support of $X^{(j)}$. Let $P_{X^{\left(
j\right) }}$ denote the law of $X^{(j)}$. Model $\mathcal{P}$ is then
defined by the set of conditions
\begin{equation}
\int g_{j}\left( z,\theta \right) \ f\left( z,\theta \right) \ dy^{\left(
j\right) }=0\quad P_{X^{( j) }}-a.s.,\quad j\in \left\{ 1,\ldots ,J\right\} ;
\label{*}
\end{equation}%
for a fixed $k$, the model $\mathcal{P}_{k}$ is defined by
\begin{equation}
\int g_{j}\left( z,\theta \right) \ f\left( z,\theta \right) \ \overline{w}%
_{s}^{j}\left( x^{\left( j\right) }\right) \ dz=0,\quad j\in \left\{
1,\ldots ,J\right\} ,\quad s\in \left\{ 1,\ldots ,k\right\} ,  \label{**}
\end{equation}
where
\begin{equation*}
\overline{w}_{s}^{j}\left( x^{\left( j\right) }\right) =E\left[ w_{s}\left(
Z\right) |X^{\left( j\right) }=x^{\left( j\right) }\right] .
\end{equation*}%
Consider now a regular parametric family $\left\{ f_{t}\right\} _{t\in
\left( -\varepsilon ,\varepsilon \right) }$ of densities satisfying $( \ref%
{*}) $, that means that there exist parameters $\theta _{t}\in \Theta $,
such that, for any $t\in \left( -\varepsilon ,\varepsilon \right) $ and $%
P_{X^{\left( j\right) }}-a.s.$,%
\begin{equation}
\int g_{j}\left( z,\theta _{t}\right) \ f_{t}\left( z,\theta _{t}\right) \
dy^{\left( j\right) }=0,\quad \forall j\in \left\{ 1,\ldots ,J\right\} .
\label{***}
\end{equation}%
Let%
\begin{eqnarray*}
\dot{\theta}&=&\left. \frac{\partial \theta _{t}}{\partial t}\right\vert
_{t=0}, \\
&& \\
s &=&\left. \partial _{t}\log f_{t}\left( Z,\theta _{0}\right) \right\vert
_{t=0},\quad S_{\theta _{0}}=\left. \partial _{\theta }\log f\left( Z,\theta
\right) \right\vert _{\theta =\theta _{0}},\quad \\
&& \\
s_{1} &=&\left. \partial _{t}\log f_{t}\left( Z,\theta _{t}\right)
\right\vert _{t=0}=s+S_{\theta _{0}}^{\prime }\ \dot{\theta}.
\end{eqnarray*}%
Here and in the following, the derivatives of the log-densities are to be
understood in the mean square sense, see Ibragimov and Has'minskii (1981),
page 64. Differentiating with respect to $t$ in $\left( \ref{***}\right) $
we obtain
\begin{equation}
E\left[ \partial_{\theta ^{\prime }}g_{j}\left( Z,\theta _{0}\right)
|X^{\left( j\right) }\right] \ \dot{\theta}\left. +\right. E\left[
g_{j}\left( Z,\theta _{0}\right) \ s_{1} (Z) \mid X^{\left( j\right) }\right]
\left. =\right. 0,\quad \forall j\left. \in \right. \left\{ 1,\ldots
,J\right\} .  \label{tangent space equations}
\end{equation}%
Since $\dot{\theta}\in \mathbb{R}^{d}$ could be arbitrary, we deduce that
for each $j \in \{ 1,\ldots ,J\}$,
\begin{equation*}
E\left[ \partial_{\theta ^{\prime }}g_{j}\left( Z,\theta _{0}\right) \mid
X^{\left( j\right) }\right] +E\left[ g_{j}\left( Z,\theta _{0}\right)
S_{\theta _{0}}^{\prime }(Z) \mid X^{\left( j\right) }\right] \left.
=\right. 0, \quad E\left[ g_{j}\left( Z,\theta _{0}\right) s(Z) \mid
X^{\left( j\right) }\right] \left. =\right. 0.
\end{equation*}
The last equation and the expression of the score functions $s_1$ suggest a
tangent space $\mathcal{T}=\mathcal{T}\left( \mathcal{P},P_{0}\right) $ of
the form%
\begin{equation}  \label{tangent space}
\mathcal{T} =\overline{\mathrm{lin}} \, S_{\theta_0 } +\left\{ s : E\left(
s^{2}\right) <\infty ,\ E\left( s\right) =0,\ E\left[ g_{j}\left( Z,\theta
_{0}\right) s(Z) \mid X^{\left( j\right) }\right] = 0, \;\; 1\leq j\leq
J\right\} .
\end{equation}
On the other hand, the tangent space $\mathcal{T}_{k}=\mathcal{T}\left(
\mathcal{P}_{k},P_{0}\right) $ corresponding to the model defined by the
equations $\left( \ref{**}\right) $ is given by vectors satisfying the
unconditional moment equations
\begin{equation}  \label{tangent space k}
E\left[ \partial_{\theta ^{\prime }}g_{j}\left( Z,\theta _{0}\right) \
\overline{w}_{r}^{j}\left( X^{\left( j\right) }\right) \right] \ \dot{\theta}%
+E\left[ g_{j}\left( Z,\theta _{0}\right) s_{1}(Z)\ \overline{w}%
_{r}^{j}\left( X^{\left( j\right) }\right) \right] = 0,
\end{equation}%
$1\leq j \leq J,$ $1\leq r \leq k.$ This yields the tangent spaces
\begin{eqnarray*}
\mathcal{T}_{k} &=&\overline{\mathrm{lin}} \, S_{\theta_0 } +\left\{ s :
E(s^{2}) <\infty ,\ E(s) =0,\ E\left[ g_{j}\left( Z,\theta _{0}\right) \
s(Z)\ \overline{w}_{r}^{j}\left( X^{\left( j\right) }\right) \right] \left.
=\right. 0,\right. \\
&&\qquad \qquad \qquad \qquad \qquad \qquad \qquad \qquad \qquad \qquad
\qquad \left. \;\;\; \forall \ 1\leq j\leq J,~\forall \ 1\leq r\leq
k\right\} ;
\end{eqnarray*}%
see for instance Example 3, section 3.2 in Bickel, Klaassen, Ritov and
Wellner (1993). Since the functions $w_{k}\left( Z\right) $, $k\in \mathbb{N}%
^{\ast }$, span $L^{2}\left( P_0 \right) $, their projections $\overline{w}%
_{k}^{j}\left( X^{\left( j\right) }\right) $ on $L^{2}\left( P_{X^{\left(
j\right) }}\right) $, $k\in \mathbb{N}^{\ast }$, will span $L^{2}\left(
P_{X^{\left( j\right) }}\right) $. Consequently, equations $\left( \ref%
{tangent space equations}\right) $ are satisfied if and only if equations $%
\left( \ref{tangent space k}\right) $ are satisfied for any $k\in \mathbb{N}%
^{\ast }$. In other words, the equivalent of the spanning condition of
Newey~(2004), see our equation (\ref{spaning cond 2}) above, is satisfied
and we can apply Lemma~\ref{key_lemma} to conclude that $I_{\theta
_{0}}=\lim\limits_{k\rightarrow \infty }I_{\theta _{0}}^{\left( k\right) }$.

The proof will be complete if we show that the tangent space $\mathcal{T}=%
\mathcal{T}\left( \mathcal{P},P_{0}\right)$ is indeed the set described in
equation (\ref{tangent space}). Consider for simplicity that $J=2$, the
general case could be handled similarly. It is quite easy to see that
equations $\left( \ref{tangent space equations}\right) $ guarantees the
inclusion ``$\subset$'' in display $\left( \ref{tangent space}\right) $. To
show the reverse inclusion, it suffices to prove that $\mathcal{T}^{\prime
}\subset \mathcal{T}$, where
\begin{eqnarray*}
\mathcal{T}^{\prime } &=&\mathcal{T}^{\prime }\left( \mathcal{P}%
,P_{0}\right) =\left\{ s:E( s^{2}) <\infty ,\ E( s ) =0,\right. \\
&&\qquad \qquad \qquad \left. E\left[ g_{1}\left( Z,\theta _{0}\right) \
s(Z) \mid X^{\left( 1\right) }\right] \left. =\right. 0,\ E\left[
g_{2}\left( Z,\theta _{0}\right) \ s(Z) \mid X^{\left( 2\right) }\right]
\left. =\right. 0\right\} .
\end{eqnarray*}%
Let $f_0$ denote the true density of the vector $Z$. Take $s\in \mathcal{T}%
^{\prime }$ and suppose for the moment that $s$ is bounded. Then, for real
numbers $t$ with sufficiently small absolute values, the functions $f_{t}=
\left( 1+t\cdot s\right) f_{0}$ are densities on $\mathcal{Z}$ and if $%
E_{f_{t}}$ denotes expectation with respect to the law defined by $f_t$,
\begin{equation*}
E_{f_{t}}\left[ g_{j}\left( Z,\theta _{0}\right) \ a\left( X^{\left(
j\right) }\right) \right] =E\left[ g_{j}\left( Z,\theta _{0}\right) \
a\left( X^{\left( j\right) }\right) \right] + t \ E\left[ g_{j}\left(
Z,\theta _{0}\right) \ s\left( Z\right) \ a\left( X^{\left( j\right)
}\right) \right] =0,
\end{equation*}%
for any square-integrable function $a( X^{( j) }) $, so that $E_{f_{t}}[
g_{j}\left( Z,\theta _{0}\right) |X^{\left( j\right) }] =0,$ $j=1,2.$
Moreover,
\begin{equation*}
\partial _{t}\left. \log f_{t}\right\vert _{t=0}=\partial _{t}\left. \log
\left( 1+t\cdot s\right) \right\vert _{t=0}=s,
\end{equation*}%
which means that the family of densities $\left\{ f_{t}\right\} _{\left\vert
t\right\vert <\varepsilon }$ defines a submodel of model $\left( \ref%
{model_I}\right) $ for which the tangent vector at $t=0$ is exactly $s$.
Next, we have to extend the argument to unbounded functions $s$. If $%
\mathcal{M}\subset L^{2}\left( P_0\right) $ is the subspace of bounded
functions of $Z$, it remains to show that $\mathcal{M\cap T}^{\prime }$ is
dense in $\mathcal{T}^{\prime }$. One may consider this step obvious since
any unbounded square integrable function can be approximated by a sequence
of bounded functions, see for instance Ai and Chen (2003), page 1838. We
argue that this well-known approximation result cannot be directly applied
to our context, as it is also the case in other contexts considered in the
efficiency bounds literature. Indeed, here we are in the following
situation: we have two infinite-dimension closed subspaces $\mathcal{T}%
_{1}^{\prime }$ and $\mathcal{T}_{2}^{\prime }$ such that $\mathcal{T}%
^{\prime }=\mathcal{T}_{1}^{\prime }\cap \mathcal{T}_{2}^{\prime }$, $%
\overline{\mathcal{M\cap T}_{1}^{\prime }}=\mathcal{T}_{1}^{\prime }$ and $%
\overline{\mathcal{M\cap T}_{2}^{\prime }}=\mathcal{T}_{2}^{\prime }$, and
we need that $\overline{\mathcal{M\cap T}^{\prime }}=\mathcal{T}^{\prime }.$
To our best knowledge, there is no general mathematical result which would
allow us to claim that $\mathcal{M\cap T}^{\prime }$ is dense in $\mathcal{T}%
^{\prime }$ without any further argument. That is why we have to provide a
proof adapted to the case we consider herein. By Assumption~$T$ and the
subsequent remark, and equation (\ref{b_i_A}), there exist two bounded
vector functions $b_{1}$ and $b_{2}$ defined like in equation (\ref{b_i_A})
such that, for $i,j\in\{1,2\}$, $i\neq j$,
\begin{equation*}  \label{fct_b_cond}
E\left( g_{i}\ b_{i}^{\prime }\ |\ X^{\left( i\right) }\right) =0 \quad
\text{and} \quad \left\Vert E^{-1}\left( g_{i}\ b_{j}^{\prime }\ |\
X^{\left( 1\right) },X^{\left( 2\right) }\right) \right\Vert _{\infty }<1,
\end{equation*}
where $g_i = g_i(Z, \theta_0)$. Here and in the sequel, the norm of a vector
(or matrix) should be understand as the sum of componentwise norms. Since $%
\mathcal{M}$ is dense in $L^{2}\left( P_0\right) $, for a fixed $s\in
\mathcal{T}^{\prime }$ there exist a sequence $\left\{ t_{n}\right\}
_{n}\subset \mathcal{M}$ such that
\begin{equation*}
\left\Vert s-t_{n}\right\Vert _{L^{2}(P_0)}\underset{n\rightarrow \infty }{%
\longrightarrow }0.
\end{equation*}%
Define%
\begin{equation*}
u_{n}=t_{n}-E\left( t_{n}g_{1}^{\prime }\ |\ X^{\left( 1\right) }\right) \
E^{-1}\left( b_{1}g_{1}^{\prime }\ |\ X^{\left( 1\right) }\right) \
b_{1}-E\left( t_{n}g_{2}^{\prime }\ |\ X^{\left( 2\right) }\right) \
E^{-1}\left( b_{2}g_{2}^{\prime }\ |\ X^{\left( 2\right) }\right) \ b_{2}.
\end{equation*}%
It is clear that we can take $\left\{ t_{n}\right\} _{n}\subset \mathcal{M}$
such that
\begin{equation*}
\left\Vert E\left(t_n g_{1}\ |\ X^{\left( 1\right) }\right) \right\Vert
_{\infty }+\left\Vert E\left( t_n g_{2}\ |\ X^{\left( 2\right) }\right)
\right\Vert_\infty <\infty
\end{equation*}
and thus $u_{n}\in \mathcal{M}$. Then%
\begin{eqnarray}
E\left( g_{1}u_{n}^{\prime }\ |\ X^{\left( 1\right) }\right) &=&\underline{%
\underline{E\left( g_{1}t_{n}^{\prime }\ |\ X^{\left( 1\right) }\right) }}-%
\underline{E\left( g_{1}b_{1}^{\prime }\ |\ X^{\left( 1\right) }\right) }\
\underline{E^{-1}\left( g_{1}b_{1}^{\prime }\ |\ X^{\left( 1\right) }\right)
}\ \underline{\underline{E\left( g_{1}t_{n}^{\prime }\ |\ X^{\left( 1\right)
}\right) }}  \notag \\
&&-E\left[ g_{1}b_{2}^{\prime }\ E^{-1}\left( g_{2}b_{2}^{\prime }\ |\
X^{\left( 2\right) }\right) \ E\left( g_{2}t_{n}^{\prime }\ |\ X^{\left(
2\right) }\right) \ |\ X^{\left( 1\right) }\right]  \notag \\
&=&- E\left[ {}\right. \underset{=0}{\underbrace{E\left( g_{1}b_{2}^{\prime
}\ |\ X^{\left( 1\right) },X^{\left( 2\right) }\right) }}\ E^{-1}\left(
g_{2}b_{2}^{\prime }\ |\ X^{\left( 2\right) }\right) \ E\left(
g_{2}t_{n}^{\prime }\ |\ X^{\left( 2\right) }\right) \ |\ X^{\left( 1\right)
}\left. {}\right]  \notag \\
&=&0,  \label{ddge}
\end{eqnarray}%
and similarly,%
\begin{equation}  \label{ddge2}
E\left( g_{2}u_{n}^{\prime }\ |\ X^{\left( 2\right) }\right) =0.
\end{equation}%
Moreover,
\begin{eqnarray*}
s-u_{n} &=&s-t_{n}+t_{n}-u_{n} \\
&=&s-t_{n}+E\left[ \left( t_{n}-s\right) g_{1}^{\prime }\ |\ X^{\left(
1\right) }\right] \ E^{-1}\left( b_{1}g_{1}^{\prime }\ |\ X^{\left( 1\right)
}\right) \ b_{1} \\
&&+E\left[ \left( t_{n}-s\right) g_{2}^{\prime }\ |\ X^{\left( 2\right) }%
\right] \ E^{-1}\left( b_{2}g_{2}^{\prime }\ |\ X^{\left( 2\right) }\right)
\ b_{2},
\end{eqnarray*}%
which entails%
\begin{eqnarray*}
\left\Vert s-u_{n}\right\Vert _{L^{2}(P_0)} &\leq &\left\Vert
s-t_{n}\right\Vert _{L^{2}(P_0)}+\left\Vert E\left[ \left( t_{n}-s\right)
g_{1}^{\prime }\ |\ X^{\left( 1\right) }\right] \right\Vert
_{L^{2}(P_0)}\cdot \left\Vert b_{1}\right\Vert _{\infty } \\
&&+\left\Vert E\left[ \left( t_{n}-s\right) g_{2}^{\prime }\ |\ X^{\left(
2\right) }\right] \right\Vert _{L^{2}(P_0)}\cdot \left\Vert b_{2}\right\Vert
_{\infty }.
\end{eqnarray*}%
Noting that
\begin{eqnarray*}
\left\Vert E\left[ \left( t_{n}-s\right) g_{1}^{\prime }\ |\ X^{\left(
1\right) }\right] \right\Vert _{L^{2}(P_0)}^{2} &=&E\left\{ E^{2}\left[
\left( t_{n}-s\right) g_{1}^{\prime }\ |\ X^{\left( 1\right) }\right]
\right\} \\
(Cauchy-Schwarz)\quad &\leq &E\left\{ E^{2}\left[ \left( t_{n}-s\right) \ |\
X^{\left( 1\right) }\right] \ E^{2}\left( g_{1}^{\prime }\ |\ X^{\left(
1\right) }\right) \right\} \\
&\leq &\left\Vert E\left( g_{1}\ |\ X^{\left( 1\right) }\right) \right\Vert
_{\infty }^{2}\ E\left\{ E^{2}\left[ \left( t_{n}-s\right) \ |\ X^{\left(
1\right) }\right] \right\} \\
(Jensen)\quad &\leq &\left\Vert E\left( g_{1}\ |\ X^{\left( 1\right)
}\right) \right\Vert _{\infty }^{2}\ E\left\{ E\left[ \left( t_{n}-s\right)
^{2}\ |\ X^{\left( 1\right) }\right] \right\} \\
&\leq &\left\Vert E\left( g_{1}\ |\ X^{\left( 1\right) }\right) \right\Vert
_{\infty }^{2}\ \left\Vert t_{n}-s\right\Vert _{L^{2}(P_0)}^{2},
\end{eqnarray*}%
we finally obtain $\left\Vert s-u_{n}\right\Vert
_{L^{2}(P_0)}\longrightarrow 0$ as $n\rightarrow\infty.$ In particular,
deduce that $E(u_n)\rightarrow 0$. Now, since all the previous equations and
inequalities involving $u_n$ hold also with $u_n$ replaced by $u_n - E(u_n)$%
, deduce that $\left\{ u_{n} - E(u_n) \right\} _{n}\subset \mathcal{M}\cap
\mathcal{T}^{\prime }, $ which implies that $s\in \overline{\mathcal{M}\cap
\mathcal{T}^{\prime }}. $ Now the proof is complete.
\end{proof}

\bigskip

In the general theory of efficiency bounds, the semiparametric Fisher
information on a finite dimension parameter in a semiparametric model is the
infimum of the Fisher information over all its parametric submodels; see for
instance Newey (1990). For models defined by conditional moment equations,
Theorem \ref{Main_Result} shows that the same semiparametric Fisher
information can be alternatively obtained as the lower limit of the
semiparametric Fisher information in a sequence of decreasing supra-models.
The main reason for this is that with such decreasing sequence of
supra-models, the `spanning condition' (\ref{spaning cond 2}) holds true.
Moreover, since $L^2 (P_0)$ is a separable Hilbert space, Theorem \ref%
{Main_Result} can be restated under the following equivalent form.

\begin{corollary}
\label{corr_main_th} Under the conditions of Theorem \ref{Main_Result},
\begin{equation*}
I_{\theta _{0}}=\sup_{b\in \mathcal{B}}I_{\theta _{0}}\left( b\right) ,
\end{equation*}%
where%
\begin{equation*}
\mathcal{B}=\left\{ \left( b_{1}\left( X^{\left( 1\right) }\right) ,\ldots
,b_{J}\left( X^{\left( J \right) }\right) \right) :~b_{j,lk}\in L^{2}\left(
P_{X^{(j)}} \right) 1\leq l\leq d, 1\leq k \leq p_{j},~1\leq j\leq J\right\},
\end{equation*}%
so that any $b = b \left( \underline{X}\right) \in \mathcal{B}$ is a $%
d\times p-$matrix with random elements, and $I_{\theta _{0}}\left( b\right) $
is the Fisher information on $\theta _{0}$ in the model defined by the
marginal moment restrictions
\begin{equation}
E\left[ b_{j}\left( X^{\left( j\right) }\right) \ g_{j}\left( Z,\theta
\right) \right] =0,\quad j\in \left\{ 1,\ldots ,J\right\} ,
\label{model IV b}
\end{equation}%
model which can also be written under the compact form $E\left[ b\left(
\underline{X}\right) \ \underline{g}\left( Z,\theta \right) \right] =0. $
\end{corollary}

\bigskip

\begin{remark}
We argue that, under further assumptions, the result of Theorem \ref%
{Main_Result} extends to the case where the unknown functions $g_{j}$ depend
also on a same unknown function $h$ of the observations and the parameter.
More precisely, when the model is defined by
\begin{equation}
E\left[ \widetilde{g}_{j}\left( Z,\theta ,h\left( Z,\theta \right) \right) \
|\ X^{\left( j\right) }\right] =0,\quad j=1,\ldots ,J,  \label{model IV}
\end{equation}%
where $\widetilde{g}_{j}:\mathbb{R}^{q}\mathbb{\times R}^{d}\mathbb{\times R}%
^{p_{h}}\rightarrow \mathbb{R}^{p_{j}}$, $j\in \left\{ 1,\ldots ,J\right\} $%
, are known. With the same notations used for defining $\underline{g}%
_{k}^{w} $, let
\begin{eqnarray*}
\underline{\widetilde{g}}_{k}^{w}\left( Z,\theta ,h\left( Z,\theta \right)
\right) &=&\underline{w}^{\left( k\right) }\left( \underline{X}\right) \;
\underline{\widetilde{g}}\left( Z,\theta ,h\left( Z,\theta \right) \right)
\in \mathbb{R}^{kp},\qquad \forall k\in \mathbb{N}^{\ast }, \\
&&
\end{eqnarray*}%
where $\underline{\widetilde{g}}=\left( \widetilde{g}_{1}^{\;\prime },\ldots
,\widetilde{g}_{J}^{\; \prime }\right)^{\prime } $ and let $\widetilde{I}%
_{\theta _{0}}^{\left( k\right) }$ be the Fisher information on $\theta _{0}$
in the model%
\begin{equation}
E\left[ \underline{\widetilde{g}}_{k}^{w}\left( Z,\theta ,h\left( Z,\theta
\right) \right) \right] =0;  \label{model V}
\end{equation}%
its expression as a solution of a variational problem can be found in
Chamberlain~(1992), Ai and Chen~(2003) or Chen and Pouzo~(2009).

Similar but more involved arguments can be invoked to show the following
result, which we state here as a conjecture: the information $\widetilde{I}%
_{\theta _{0}}$ on $\theta _{0}$ at $P_{0}$ in model $\left( \ref{model IV}%
\right) $ is given by
\begin{equation*}
\widetilde{I}_{\theta _{0}}=\lim_{k\rightarrow \infty }\widetilde{I}_{\theta
_{0}}^{\left( k\right) },
\end{equation*}%
where $\widetilde{I}_{\theta _{0}}^{\left( k\right) }$ is the Fisher
information on $\theta _{0}$ in model $\left( \ref{model V}\right) $.
\end{remark}

\bigskip

\section{Efficient estimation}

\label{Section : Efficient estimation}

To simplify the presentation, let us take $J=2$. To obtain an efficient estimator, a common way is to solve $\theta$ from the
efficient score equations; see van der Vaart (1998), section 25.8. By
definition, the efficient score is the componentwise projection of the score
$S_{\theta_0}$ on the orthogonal complement of the tangent space $\mathcal{T}%
=\mathcal{T}(\mathcal{P},P_0)$ defined in equation (\ref{tangent space}). In
the projection of $S_{\theta_0}$ on $\mathcal{T}^{\perp}$ only the nonparametric
part of the tangent space matters. Moreover, the projection of $S_{\theta_0}$
is componentwise. It is then common practice in the literature to identify $%
\mathcal{T}$ with the subspace of $\left\{ L^{2}\left( P_0\right) \right\}
^{ d} =\bigoplus_{k=1}^d L^{2}\left( P_0\right)$ obtained as the $d-$fold
cartesian product of the nonparametric part of $\mathcal{T}$. Here the
direct sum of Hilbert spaces is considered with the usual inner product $%
\langle (\phi_1,\cdots,\phi_d), (\psi_1,\cdots,\psi_d)\rangle = \langle
\phi_1,\psi_1 \rangle + \cdots + \langle \phi_d,\psi_d \rangle.$ Therefore
we will slightly change our notation for the tangent spaces. More precisely,
let us define
\begin{eqnarray*}
\mathcal{T} &=&\left\{ s\in \bigoplus_{k=1}^d L^{2}\left( P_0\right):\
E\left( s\right) =0,\ E\left( g_{i}(Z,\theta_0)s^{\prime }(Z)\ |\ X^{\left(
i\right) }\right) =0,\ i=1,2\right\} \\
&& \\
&=&\mathcal{T}_{1}\cap \mathcal{T}_{2},
\end{eqnarray*}%
where, for $i=1,2,$
\begin{equation*}
\mathcal{T}_{i} =\left\{ s\in \bigoplus_{k=1}^d L^{2}\left( P_0\right):\
E\left( s\right) =0,\ E\left( g_{i}(Z,\theta_0)s^{\prime }(Z)\ |\ X^{\left(
i\right) }\right) =0\right\} ,
\end{equation*}
so that%
\begin{equation*}
\mathcal{T}_{i}^{\perp } =\left\{ s\in \bigoplus_{k=1}^d L^{2}\left(
P_0\right):\ s(Z) =a_{i}\left( X^{\left( i\right) }\right) \
g_{i}(Z,\theta_0) \right\} .
\end{equation*}
Clearly, $\mathcal{T}^{\perp } =\overline{ \mathcal{T}_{1}^{\perp } + \mathcal{T}_{2}^{\perp }}.$

In general, the projection of $S_{\theta_0}$ on $\mathcal{T}^\perp$ is not explicit. To approximate this projection and to further build an asymptotically efficient estimator for model (\ref{model_I}), we use the
iterative (``backfitting'' or successive approximation) procedure considered in Theorem~A.4.2 of Bickel,
Klaassen, Ritov and Wellner (1993), page 438; BKRW hereafter. Let $H_{i}=%
\mathcal{T}_{i}^{\perp },$ $g_i = g(Z,\theta_0),$ $i=1,2,$ and let $%
E(\partial_\theta g_i^\prime) $ be the transposed of the matrix $%
E(\partial_{\theta^\prime} g_i) $ defined in equation (\ref{eq_biz}). The
steps of the procedure we propose are the following~:

\begin{enumerate}
\item Set $m=0$. Take $a_{1}^{\left( 0\right) }=0$.

\item Put $m=m+1$. Calculate
\begin{equation*}
\overline{S}_{\theta_0 }^{\left( m\right) }=a_{1}^{\left( m\right) }\left(
X^{\left( 1\right) }\right) \ g_{1}+a_{2}^{\left( m\right) }\left( X^{\left(
2\right) }\right) \ g_{2}
\end{equation*}%
where
\begin{eqnarray*}
a_{1}^{\left( m\right) }\left( X^{\left( 1\right) }\right) &=&a_{1}^{\left(
m\right) }\left( X^{\left( 1\right) },\theta_0\right) = -E\left( \partial
_{\theta }g_{1}^{\prime }\ |\ X^{\left( 1\right) }\right) \ V^{-}\left(
g_{1}\ |\ X^{\left( 1\right) }\right) \\
&& \\
&&\hspace{-2.5cm}+E\left[ E\left( \partial _{\theta }g_{2}^{\prime }\ |\
X^{\left( 2\right) }\right) \ V^{-}\left( g_{2}\ |\ X^{\left( 2\right)
}\right) \ g_{2}\ g_{1}^{\prime }\ |\ X^{\left( 1\right) }\right] \
V^{-}\left( g_{1}|X^{\left( 1\right) }\right) \\
&& \\
&&\hspace{-2.5cm}+E\left[ E\left[ a_{1}^{\left( m-1\right) }\left( X^{\left(
1\right) }\right) \ g_{1}\ g_{2}^{\prime }\ |\ X^{\left( 2\right) }\right] \
V^{-}\left( g_{2}|X^{\left( 2\right) }\right) \ g_{2}\ g_{1}^{\prime }\ |\
X^{\left( 1\right) }\right] \ V^{-}\left( g_{1}|X^{\left( 1\right) }\right)
\end{eqnarray*}%
and%
\begin{eqnarray*}
a_{2}^{\left( m\right) }\left( X^{\left( 2\right) }\right) &=& a_{2}^{\left(
m\right) }\left( X^{\left( 2\right) },\theta_0\right) = -E\left( \partial
_{\theta }g_{2}^{\prime }\ |\ X^{\left( 2\right) }\right) \ V^{-}\left(
g_{2}\ |\ X^{\left( 2\right) }\right) \\
&&-E\left[ a_{1}^{\left( m\right) }\left( X^{\left( 1\right) }\right) \
g_{1}\ g_{2}^{\prime }\ |\ X^{\left( 2\right) }\right] \ V^{-}\left(
g_{2}|X^{\left( 2\right) }\right) .
\end{eqnarray*}

\item Repeat from step~2 till the convergence of $\overline{S}_{\theta_0
}^{\left( m\right) }$.
\end{enumerate}

Let $\Pi \left( s|\mathcal{S}\right) $ denote the (componentwise) projection
of a vector $s\in \bigoplus_{k=1}^d L^{2}\left( P_0\right)$ on a subspace $%
\mathcal{S\subset }\bigoplus_{k=1}^d L^{2}\left( P_0\right)$. Theorem~A.4.2
(A) from BKRW directly yields the following result.

\begin{lemma}
\label{Iterative algorithm} Assume that the conditions of Theorem \ref%
{Main_Result} hold true. When $m\rightarrow\infty$,
\begin{equation*}
\overline{S}_{\theta_0 }^{\left( m\right) } = a_{1}^{( m) }( X^{(1) }) g_{1}
+ a_{2}^{( m) }( X^{( 2) }) g_{2}\longrightarrow \overline{S}_{\theta_0 } =
\Pi \left( S_{\theta_0 }|\mathcal{T}^{\perp }\right) =\Pi \left( S_{\theta_0
}|\overline{ H_{1}+H_{2}}\right)
\end{equation*}%
in $\bigoplus_{k=1}^d L^{2}\left( P_0\right)$, where $g_i = g(Z,\theta_0),$ $%
i=1,2.$
\end{lemma}

Let us point out that even if Lemma \ref{Iterative algorithm} guarantees the convergence of the iterations $\overline{S}_{\theta_0 }^{\left( m\right) }$, it is not necessarily true that the sequences
$a_{1}^{\left( m\right) }\left( X^{\left( 1\right) }\right) g_1$ and $%
a_{2}^{\left( m\right) }\left( X^{\left( 2\right) }\right) g_2$ converge.
Sufficient mild conditions are provided in Theorem~A.4.2 (C) of BKRW, that
are
\begin{equation}
\overline{S}_{\theta _{0}}=\Pi \left( S_{\theta _{0}}|\mathcal{T}^{\perp
}\right) =a_{1}^{\ast }\left( X^{\left( 1\right) }\right) \cdot
g_{1}+a_{2}^{\ast }\left( X^{\left( 2\right) }\right) \cdot g_{2}\in
\mathcal{T}_{1}^{\perp }+\mathcal{T}_{2}^{\perp }  \label{a_4_2c}
\end{equation}%
with $a_{1}^{\ast }\left( X^{\left( 1\right) }\right) \cdot g_{1}\in
\mathcal{T}_{1}^{\perp }\cap \left( \mathcal{T}_{1}^{\perp }\cap \mathcal{T}%
_{2}^{\perp }\right) ^{\perp }\subset \mathcal{T}_{1}^{\perp }.$ Moreover,
by Proposition A.4.1 of BKRW, condition $\left( \ref{a_4_2c}\right) $ is
equivalent with the existence of a solution $a_{1}^{\ast }g_1$ and $a_{2}^{\ast
}g_2$ for the system
\begin{equation}
\left\{
\begin{array}{c}
a_{1}^{\ast }\left( X^{\left( 1\right) }\right) \ g_{1}=\rho _{1}-E\left[
a_{2}^{\ast }\left( X^{\left( 2\right) }\right) \ g_{2}\ g_{1}^{\prime }\ |\
X^{\left( 1\right) }\right] \ V^{-}\left( g_{1}|X^{\left( 1\right) }\right)
\ g_{1} \\
\\
a_{2}^{\ast }\left( X^{\left( 2\right) }\right) \ g_{2}=\rho _{2}-E\left[
a_{1}^{\ast }\left( X^{\left( 1\right) }\right) \ g_{1}\ g_{2}^{\prime }\ |\
X^{\left( 2\right) }\right] \ V^{-}\left( g_{2}|X^{\left( 2\right) }\right)
\ g_{2},%
\end{array}%
\right.  \label{RMD 2}
\end{equation}%
where
\begin{eqnarray*}
\rho _{i} =\rho _{i} (Z,\theta_0)&:=&\Pi \left( S_{\theta_0 }|\mathcal{T}%
_{i}^{\perp }\right) =E\left( S_{\theta_0 }g_{i}^{\prime }\ |\ X^{\left(
i\right) }\right) \ V^{-}\left( g_{i}\ |\ X^{\left( i\right) }\right) \ g_{i}
\\
&=&-E\left( \partial _{\theta }g_{i}^{\prime }\ |\ X^{\left( i\right)
}\right) \ V^{-}\left( g_{i}\ |\ X^{\left( i\right) }\right) \ g_{i}.
\end{eqnarray*}
(A careful inspection of the proof of Proposition A.4.1 of BKRW shows that
condition $H_{1}+H_{2}=\mathcal{T}_{1}^{\perp }+\mathcal{T}_{2}^{\perp }$ is
a closed subspace is not necessary for deriving that result, since what is
really used in their proof is the relation $H_{1}^{\perp }\cap H_{2}^{\perp
}=\left( H_{1}+H_{2}\right) ^{\perp }$). If in addition the system $\left( %
\ref{RMD 2}\right) $ has a unique solution, the backfitting algorithm above
is nothing but a convergent iterative procedure for finding it.

In applications, a convenient way to check uniqueness is to prove a
contraction property. This is the case for instance if $\mathcal{T}%
_{1}^{\perp }\cap \mathcal{T}_{2}^{\perp }=\left\{ 0\right\} $, \ which in
our framework holds if
\begin{equation*}
E\left( g_{1}\ g_{2}^{\prime }\ |\ X^{\left( 1\right) },X^{\left( 2\right)
}\right) =0
\end{equation*}%
(in the sequential case, this can be achieved by writing the initial system
in an equivalent form satisfying the orthogonal condition above; see
subsection~\ref{seqmom}).

In the general case where $\mathcal{T}_{1}^{\perp }\cap \mathcal{T}%
_{2}^{\perp }\neq \left\{ 0\right\} $ the system $\left( \ref{RMD 2}\right) $
rewritten as in Proposition A.4.1 of BKRW under the form
\begin{equation*}
\left\{
\begin{array}{c}
h_{1}^{\ast }=\Pi \left( S_{\theta_0 }-h_{2}^{\ast }|\mathcal{T}_{1}^{\perp
}\right) \\
\\
h_{2}^{\ast }=\Pi \left( S_{\theta_0 }-h_{1}^{\ast }|\mathcal{T}_{2}^{\perp
}\right) ,%
\end{array}%
\right.
\end{equation*}%
does not necessarily have the contraction property. In our problem $%
h_{1}^{\ast }=a_{1}^{\ast } g_{1}$ and $h_{2}^{\ast }=a_{2}^{\ast } g_{2}$
with $g_{1}$ and $g_{2}$ given. Hence it suffices to check a contraction
property for $a_{1}^{\ast }g_1$ and $a_{2}^{\ast }g_2$ or some given
transformations of them.
We will see in subsection \ref{sec_appli_missing} that in the regression-like models with missing data framework, see  Robins, Rotnitzky, Zhao (1994), the
equations $\left( \ref{RMD 2}\right) $ lead to a contraction property for
some given transformations of $a_{1}^{\ast }g_1$ and $a_{2}^{\ast }g_2$.

The ``backfitting'' algorithm we proposed above involves $\theta_0$ that is
unknown. In practice one can use the following steps: (i) build $\widetilde
\theta_n$ a $\sqrt{n}-$consistent estimator of $\theta _{0}$, for instance
the \emph{smooth minimum distance estimator (SMD)} like in Lavergne and
Patilea (2008); (ii) estimate nonparametrically $a_{1}^{\left(
m^\star\right) }$ and $a_{2}^{\left( m^\star\right) }$ the solution of the
``backfitting'' algorithm obtained after, say, $m^\star$ iterations using $%
\widetilde \theta_n$ instead of $\theta_0$; and (iii) construct an efficient
(classical GMM or SMD) estimator $\widehat{\theta }^{\left( m^\star\right) }$
based on the approximate efficient score equations $E\left( \widehat{\overline{S}}_{\theta }\right) =0, $
where
\begin{equation*}
\widehat{\overline{S}}_{\theta }=\widehat{a}_{1}^{\left( m^\star\right) }(
X^{\left( 1\right) },\widetilde\theta_n) \ g_{1}(Z,\theta) + \widehat{a}%
_{2}^{\left( m^\star\right) }( X^{\left( 2\right) },\widetilde\theta_n) \
g_{2}(Z,\theta) ,
\end{equation*}
and $\widehat{a}_{i}^{\left( m^\star\right) }( X^{\left( i\right)
},\widetilde\theta_n)$ are nonparametric estimates of ${a}_{i}^{\left(
m^\star\right) }( X^{\left( i\right) },\theta_0),$ $i=1,2.$

\section{Applications}

\label{sec_applis_b}

In this section we illustrate the utility of our theoretical results for two
general classes of models: sequential (nested) conditional models and regression-like
models with missing data. The general results in sections \ref{main_result_sec} and \ref{Section : Efficient estimation} above allow us:
(a) to complete a semiparametric efficiency bound result of Chamberlain
(1992b); and (b) to generalize the mean regression with missing data setting
of Robins, Rotnitzky and Zhao (1994) and Tan (2011) to more general moment conditions,
which includes for example quantile regressions.

\subsection{Sequential conditional moments\label{seqmom}}

Important cases where equations $\left( \ref{RMD 2}\right) $ have an
explicit solution are the cases where  $\sigma \left( X^{\left(
1\right) }\right) \subset \sigma \left( X^{\left( 2\right) }\right) $ holds
true. In the case $J=2$, the model $E( g_{j} (Z,\theta) \mid X^{( j) }) =0$, $j=1,2,$ defined in (\ref{model_I})
can be equivalently written under the form%
\begin{equation}
\left\{
\begin{array}{l}
E\left( \widetilde{g}_{1}(Z,\theta)  \mid X^{\left( 1\right) }\right) =0 \\
E\left( g_{2} (Z,\theta)  \mid X^{\left( 2\right) }\right) =0,%
\end{array}%
\right.   \label{rezt}
\end{equation}%
where
\begin{equation*}
\widetilde{g}_{1}(Z,\theta) =g_{1} (Z,\theta)  -E\left( g_{1} (Z,\theta_0) \ g_{2}^{\prime } (Z,\theta_0) \mid X^{\left(
2\right) }\right) \ V^{-1}\left( g_{2}(Z,\theta_0)  \mid X^{\left( 2\right) }\right) \
g_{2} (Z,\theta) .
\end{equation*}%
Here we suppose that $V\left( g_{1}(Z,\theta_0) \mid X^{( 1 ) }\right)$ and $V\left( g_{2}(Z,\theta_0)  \mid X^{( 2) }\right)$ are invertible and this guarantees that $\theta_0$ is also identified by the equations (\ref{rezt}). Recall that $g_i$ is a short notation for $g_i(Z,\theta_0)$ and similarly let $\widetilde g_i$ replace $\widetilde g_i(Z,\theta_0)$.

Notice that  $\widetilde{g}_{1}$ is the residual of the projection of $g_{1}$ on $%
g_{2}$ with respect to $\sigma \left( X^{\left( 2\right) }\right) $ and $%
E\left( \widetilde{g}_{1}\ g_{2}^{\prime }\mid X^{\left( 2\right) }\right)
=0.$ Let $\widetilde{\mathcal{T}}_{1}$ be the tangent space of the model
defined by the first equation in (\ref{rezt}). By the definition of $%
\widetilde{g}_{1}$, it is quite clear that condition $\widetilde{\mathcal{T}}%
_{1}^{\perp }\cap \mathcal{T}_{2}^{\perp }=\left\{ 0\right\} $ holds true.
Next, multiplying the $i$th equation in (\ref{RMD 2}) by $g_i$, taking conditional expectation given $X^{(i)}$ and finally multiplying by $V^{-1}(g_i\mid X^{(i)})$, $i=1,2$, the  system (\ref{RMD 2}) corresponding to model $\left( \ref{rezt}\right) $
becomes
\begin{equation}\label{zert}
\left\{
\begin{array}{l}
\widetilde{a}_{1}^{\ast }\left( X^{\left( 1\right) }\right) =-E\left(
\partial _{\theta }\widetilde{g}_{1}^{\;\prime }\ |\ X^{\left( 1\right)
}\right) \ V^{-1}\left( \widetilde{g}_{1}|\ X^{\left( 1\right) }\right)  \\
\qquad \qquad \qquad \qquad -E\left( \widetilde{a}_{2}^{\ast }\left(
X^{\left( 2\right) }\right) \cdot g_{2}\ \widetilde{g}_{1}^{\;\prime }\ |\
X^{\left( 1\right) }\right) \ V^{-1}\left( \widetilde{g}_{1}|\ X^{\left(
1\right) }\right)  \\
\widetilde{a}_{2}^{\ast }\left( X^{\left( 2\right) }\right) =-E\left(
\partial _{\theta }g_{2}^{\prime }\ |\ X^{\left( 2\right) }\right) \
V^{-1}\left( g_{2}\ |\ X^{\left( 2\right) }\right)  \\
\qquad \qquad \qquad \qquad -E\left( \widetilde{a}_{1}^{\ast }\left(
X^{\left( 1\right) }\right) \cdot \widetilde{g}_{1}\ g_{2}^{\prime }\ |\
X^{\left( 2\right) }\right) \ V^{-1}\left( g_{2}\ |\ X^{\left( 2\right)
}\right) .%
\end{array}%
\right.
\end{equation}%
Since by definition
$
E( \widetilde{a}_{2}^{\ast }( X^{( 2) })
g_{2}\ \widetilde{g}_{1}^{\;\prime }\ |\ X^{( 1) }) =E[
\widetilde{a}_{2}^{\ast }( X^{( 2) })  E(
g_{2}\ \widetilde{g}_{1}^{\;\prime }\mid X^{( 2) }) \ |\
X^{( 1) }] =0
$
and
$
E\left( \widetilde{a}_{1}^{\ast }\left( X^{\left( 1\right) }\right)
\widetilde{g}_{1}\ g_{2}^{\prime }\ |\ X^{\left( 2\right) }\right) =
\widetilde{a}_{1}^{\ast }\left( X^{\left( 1\right) }\right)  E\left(
\widetilde{g}_{1}\ g_{2}^{\prime }\ |\ X^{\left( 2\right) }\right) =0
$
we obtain
\begin{equation}
\left\{
\begin{array}{l}
\widetilde{a}_{1}^{\ast }( X^{( 1) }) = - E(
\partial _{\theta }\widetilde{g}_{1}^{\;\prime }\ |\ X^{( 1)
}) \ V^{-1}( \widetilde{g}_{1}\ |\ X^{( 1) })
\\
\widetilde{a}_{2}^{\ast }( X^{( 2) }) = - E (
\partial _{\theta }g_{2}^{\prime }\ |\ X^{( 2) })
V^{-1}( g_{2}\ |\ X^{( 2) })
 .%
\end{array}%
\right.
\label{RMD 3}
\end{equation}%
($
E(\partial_\theta \widetilde g_i^{\; \prime}) $ denotes the transposed of the matrix $%
E(\partial_{\theta^\prime} \widetilde g_i) .$) The efficient score $\overline{S}_{\theta _{0}}$ can then be written as%
\begin{eqnarray*}
\overline{S}_{\theta _{0}} &=&\widetilde{a}_{1}^{\ast }\left( X\right) \cdot
\widetilde{g}_{1}+\widetilde{a}_{2}^{\ast }\left( X\right) \cdot g_{2} \\
&=&-E\left( \partial _{\theta }\widetilde{g}_{1}^{\;\prime }\ |\ X^{\left(
1\right) }\right) \ V^{-1}\left( \widetilde{g}_{1}\ |\ X^{\left( 1\right)
}\right) \ \widetilde{g}_{1}-E\left( \partial _{\theta }g_{2}^{\prime }\ |\
X^{\left( 2\right) }\right) \ V^{-1}\left( g_{2}\ |\ X^{\left( 2\right)
}\right) \ g_{2}.
\end{eqnarray*}

In the particular case where $X^{\left( 1\right) }=X^{\left( 2\right) }=X$,
\begin{eqnarray*}
\overline{S}_{\theta _{0}} &=&\widetilde{a}_{1}^{\ast }\left( X\right) \cdot
\widetilde{g}_{1}+\widetilde{a}_{2}^{\ast }\left( X\right) \cdot g_{2} \\
&=&-E\left( \partial _{\theta }\widetilde{g}_{1}^{\;\prime }\ |\ X\right) \
V^{-1}\left( \widetilde{g}_{1}\ |\ X^{\left( 1\right) }\right) \ \widetilde{g%
}_{1}-E\left( \partial _{\theta }g_{2}^{\prime }\ |\ X\right) \ V^{-1}\left(
g_{2}\ |\ X\right) \ g_{2} \\
&=&
\left(
\begin{array}{c}
-E\left( \partial _{\theta }\widetilde{g}_{1}^{\;\prime }\ |\ X\right) \\
-E\left( \partial _{\theta }g_{2}^{\prime }\ |\ X\right)
\end{array}%
\right)^\prime \ V^{-1}\left( \left(
\begin{array}{c}
\widetilde{g}_{1} \\
g_{2}
\end{array}%
\right) \mid X\right) \ \left(
\begin{array}{c}
\widetilde{g}_{1} \\
g_{2}%
\end{array}%
\right)\\
&=&
-E\left( \partial _{\theta }g^{\prime }\ C^{\prime }\left( X\right) \mid
X\right)%
 V^{-1}\left( C\left( X\right)  g \mid X\right) \ C\left( X\right)
 g \\
&=&-E\left( \partial _{\theta }g^{\prime }\ |\ X\right) \ V^{-1}\left( g \mid
X\right)  g,
\end{eqnarray*}
where $g^{\prime }=\left( g_{1}^{\prime }\ g_{2}^{\prime }\right) $ and
\begin{equation*}
C\left( X\right) =\left(
\begin{array}{ccc}
I &  & -E\left( g_{1}\ g_{2}^{\prime }\mid X\right)  V^{-1}\left( g_{2}\ |\
X\right) \\
0 &  & I%
\end{array}%
\right)
\end{equation*}%
is a nonsingular random matrix. This expression of the efficient score directly yields the efficiency bound derived in
Chamberlain (1987).

Another important particular case of formulae $\left( \ref{RMD 3}\right) $
is provided by models defined by sequential conditional moments; see  Chamberlain
(1992b), Ai and Chen (2009). Taking $X^{\left( 1\right) }=X_{1}$ and $%
X^{\left( 2\right) }=\left( X_{1}^{\prime },X_{2}^{\prime }\right) ^{\prime
} $, one obtains
\begin{eqnarray*}
\overline{S}_{\theta _{0}} &=&\widetilde{a}_{1}^{\ast }\left( X\right) \cdot
\widetilde{g}_{1}+\widetilde{a}_{2}^{\ast }\left( X\right) \cdot g_{2} \\
&& \\
&=&-E\left( \partial _{\theta }\widetilde{g}_{1}^{\;\prime }\ |\ X_{1}\right)
\ V^{-1}\left( \widetilde{g}_{1}\ |\ X_{1}\right) \ \widetilde{g}%
_{1}-E\left( \partial _{\theta }g_{2}^{\prime }\ |\ X_{1},X_{2}\right) \
V^{-1}\left( g_{2}\ |\ X_{1},X_{2}\right) \ g_{2}.
\end{eqnarray*}%
Let us point that Chamberlain (1992b) only proves this result for discrete
distributions and Ai and Chen (2009) obtain the result in a more general
framework (allowing for unknown infinite dimensional parameters in the equations
defining the model) but under slightly more restrictive assumptions than in our setting.\footnote{Ai and Chen (2009) implicitly require that the class $\mathcal{G}$ appearing in their Assumption A in the Mathematical Appendix is the same for each value of their model parameter $\alpha$. This variation independent parametrization assumption represents an additional restriction that is unnecessary in our approach. See also van der Laan and Robins (2003), page 18, for some lucid comments on the existence of a variation independent parametrization. }

\subsection{Regression-like models  with missing data}
\label{sec_appli_missing}
Consider now a regression-like model defined by the equations
\begin{equation}
E\left[ \rho \left( Y,X^{\ast },\alpha\right) \ |\ X^{\ast }\right] =0,
\label{RMD 4}
\end{equation}%
where $\rho(\cdot,\cdot,\cdot)$ is some measurable vector-valued function, $\alpha$ is a (finite-dimension) vector of parameters, and
the vector $\left( Y^\prime,X^{\ast \;\prime}\right) =\left( Y^\prime,X^\prime ,V^\prime \right) $ is not
always completely observed. We also assume that a non-missing indicator $\delta$ and
some other variable $V^{0}$
are always observed.  In the following examples we consider two random missingness mechanisms considered respectively by Tan~(2011) and Robins, Rotnitzky and Zhao (1994).
\begin{example}\label{rrz_exp}
\begin{enumerate}
\item[(i)] The vector $Y$ is observed iff $\delta =1;$

\item[(ii)] The vector $W=\left(
\begin{array}{c}
X^{\ast } \\
V^{0}%
\end{array}%
\right) $ is always observed and we have
\begin{equation}
P\left( \delta =1\ |\ Y,W\right) =P\left( \delta =1\ |\ W\right) =\pi \left(
W\right) .  \label{RMD 5}
\end{equation}
\end{enumerate}
\end{example}

\begin{example} \label{tan_exp}
\begin{enumerate}
\item[(i)] Let $X^{\ast }=\left(
\begin{array}{c}
X \\
V%
\end{array}%
\right) $ where $X$ is observed iff $\delta =1;$

\item[(ii)] The vector $W=\left(
\begin{array}{c}
Y \\
V \\
V^{0}%
\end{array}%
\right) $ is always observed and we have
\begin{equation}
P\left( \delta =1\ |\ X,W\right) =P\left( \delta =1\ |\ W\right) =\pi \left(
W\right) .  \label{RMD 6}
\end{equation}
\end{enumerate}
\end{example}

Let $\alpha_0$ be the true value of the parameter identified by the model (\ref{RMD 4}). The equation $\left( \ref{RMD 4}\right) $ and each of $\left( \ref{RMD 5}\right) $
or $\left( \ref{RMD 6}\right) $ imply
\begin{equation}
E\left[ \frac{\delta }{\pi \left( W\right) }\ \rho \left( Y,X^{\ast },\alpha
_{0}\right) \ |\ X^{\ast }\right] =0.  \label{RMD 7}
\end{equation}%
We can consider this equation at the observational level even for
missing $X^{\ast }$, since for missing values of
$X^{\ast }$ we have $\delta = 0$ which renders the equation   noninformative.
Note also that $\left( \ref{RMD 5}\right) $ and $\left( \ref{RMD 6}\right) $
can be written under the unified form
\begin{equation*}
P\left( \delta =1\ |\ Y,X^{\ast },W\right) =\pi \left( W\right) .
\end{equation*}
Therefore, at the observational level, with any of the two examples we obtain  a model like
\begin{equation}
\left\{
\begin{array}{l}
E\left[ \dfrac{\delta }{\pi \left( W\right) }\ \rho \left( Y,X^{\ast
},\alpha _{0}\right) \ |\ X^{\ast }\right] =0 \\
\\
E\left[ \dfrac{\delta }{\pi \left( W\right) }-1\ |\ W\right] =0.%
\end{array}%
\right.  \label{RMD 8}
\end{equation}%
Moreover, like in Graham~(2011, footnote 8, page 442), it can be shown that at the
observational level,  a model given by equation $\left( \ref{RMD 4}\right) $
and any of the missing data mechanism described in Example \ref{rrz_exp} or Example \ref{tan_exp} is equivalent
to the model defined by $\left( \ref{RMD 8}\right) $.

With our notation,  $Z$ is the vector built as the union of all the variables contained in $Y$, $X^{\ast }$, $W$ and $\delta$,
$\theta =\alpha $, $g_{1}(Z,\theta)=\{\delta/\pi
\left( W\right)\}  \rho \left( Y,X^{\ast },\alpha \right) $, $g_{2}(Z,\theta)=\{\delta/\pi
\left( W\right)\}-1$, $X^{\left( 1\right) }=X^{\ast }$ and $%
X^{\left( 2\right) }=W$. Let $\rho$ be a short for $\rho \left( Y,X^{\ast },\alpha
_{0}\right)$.
Then the functions $a_{1}^{\ast }$ and $a_{2}^{\ast }$
defining the efficient score are given by the following equations obtained (see also equations (\ref{zert}))
from equations $\left( \ref{RMD 2}\right)$~:
\begin{eqnarray*}
a_{1}^{\ast }\left( X^{\ast }\right) &=&a_{1}^{\ast }\left( X^{\left(
1\right) }\right) \\
&=&-E\left( \partial _{\alpha }\rho ^{\prime }\ |\ X^{\ast }\right) \
E^{-1}\left( \dfrac{1}{\pi \left( W\right) }\ \rho \ \rho ^{\prime }\ |\
X^{\ast }\right) \\
&&\qquad +E\left\{ E\left[ a_{1}^{\ast }\left( X^{\ast }\right) \ \rho \ |\ W%
\right] \ \dfrac{1-\pi \left( W\right) }{\pi \left( W\right) }\ \rho
^{\prime }\ |\ X^{\ast }\right\} \ E^{-1}\left( \dfrac{1}{\pi \left(
W\right) }\ \rho \ \rho ^{\prime }\ |\ X^{\ast }\right) ;
\end{eqnarray*}%
\begin{eqnarray*}
a_{2}^{\ast }(W)
&=& a_{2}^{\ast }\left( X^{\left( 2\right) }\right)\\ &=&E\left[ a_{1}^{\ast
}\left( X^{\ast }\right) \rho\ \frac{\delta }{\pi \left( W\right) }
\left( \dfrac{\delta }{\pi \left( W\right) }-1\right) \ |\ W\right] \ E^{-1}%
\left[ \left( \dfrac{\delta }{\pi \left( W\right) }-1\right) ^{2}\ |\ W%
\right] \\
&=&-E\left[ a_{1}^{\ast }\left( X^{\ast }\right) ~\rho \mid W \right] .
\end{eqnarray*}%
In the particular case where $\rho =\rho \left( Y,X^{\ast },\alpha
_{0}\right) =Y-g\left( X^{\ast },\alpha _{0}\right) $ and the selection
probability $\pi \left( W\right) $ is known, these are exactly the
equations obtained in Robins, Rotnitzky and Zhao (1994). They showed that for the
regression case, the equation for $a_{1}^{\ast }$ corresponds to a
contraction (see the proof of their Proposition 4.2). In subsection \ref{rrz_contrac} in the Appendix we show that such a contraction property holds for a more general $\rho $. Hence we could include in our framework further interesting examples, \emph{e.g.} quantile regressions. The contraction property allows to solve the equations in $a_{1}^{\ast }( X^{\ast })$ and $a_{2}^{\ast }(W)$ by successive approximations.

Let us consider the extended framework where  the selection probability is known up to an unknown  finite dimension parameter $\gamma_{0}$, that is
\begin{equation*}
P\left( \delta =1\ |\ W\right) =\pi \left( W,\gamma _{0}\right),
\end{equation*}
(see also Robins, Rotnitzky and Zhao (1994), equation (18)). In subsection \ref {rrz_selec_prob} in the Appendix we show that the efficiency score for $\alpha_{0}$ has the same expression regardless the selection probability function $\pi$ is given or depends on the unknown parameter $\gamma_0$. Thus, we extend a result of Robins, Rotnitzky and Zhao (1994), see also  Tan~(2011), obtained in the particular case of mean regressions.

Let us close this section with a remark. Robins, Rotnitzky and Zhao (1994) considered the case where missingness arises only in covariables $X^\star$ (that is also the case considered in our Example \ref{tan_exp}) and derived the efficient score equations.  Tan~(2011) obtained formally the same equations with missing regressors \emph{and} missing responses (the case corresponding to our Example \ref{rrz_exp}) using the corresponding definition of $W$.
However, there is an important difference between the Examples \ref{rrz_exp} and \ref{tan_exp}. In the possibly missing responses case we have $\sigma \left( X^{\ast }\right) \subset \sigma \left( W\right) $, so that Example \ref{rrz_exp} falls in the sequential conditional moments framework where the solutions for $a_{1}^{\ast }$ and $a_{2}^{\ast }$ are explicit. Such explicit solutions are \emph{no longer} available in the framework considered by Robins, Rotnitzky and Zhao (1994) and in our Example \ref{tan_exp}.

\qquad

\qquad

\section{Appendix}

\subsection{Additional proofs}

\begin{proof}[Proof of Lemma \protect\ref{key_lemma}]
By definition (van der Vaart (1998), pp 363), there exists a continuous
linear map $\dot{\psi}:L^{2}\left( P_{0}\right) \rightarrow \mathbb{R}^{d}$
such that for any $g\in \mathcal{\dot{P}}_{P_{0}}\subset L^{2}\left(
P_{0}\right) $ and a submodel $\left( -\varepsilon ,\varepsilon \right) \ni
t\mapsto P_{t}$ with score function $g$,
\begin{equation*}
\frac{\psi \left( P_{t}\right) -\psi \left( P_{0}\right) }{t}\underset{%
t\rightarrow 0}{\longrightarrow }\dot{\psi}\left( g\right) .
\end{equation*}%
By the Riesz representation theorem, there exists a unique $d-$dimension
vector-valued function having the components in $L^{2}\left( P_{0}\right) $
such that $\dot{\psi}\left( h\right) =E_{P_{0}}\left( \overline{\psi }%
h\right) $ for every $h\in L^{2}\left( P_{0}\right) $. In particular,%
\begin{equation*}
\dot{\psi}\left( g\right) =E_{P_{0}}\left( \overline{\psi }g\right) =\int
\overline{\psi }gdP_{0},\qquad \forall g\in \mathcal{\dot{P}}_{P_{0}}\subset
L^{2}\left( P_{0}\right) \text{.}
\end{equation*}%
Let $\widetilde{\psi }$ and $\widetilde{\psi }_{k}$ denote the elements of $%
\left[ L^{2}\left( P_{0}\right) \right] ^{d}$ obtained by componentwise
projections of $\overline{\psi }$ on the tangent spaces $\mathcal{T}\subset
L^{2}\left( P_{0}\right) $ and $\mathcal{T}_{k}\subset L^{2}\left(
P_{0}\right) $, respectively. The Fisher information matrices on $\theta
_{0}=\psi \left( P_{0}\right) $ in the models $\mathcal{P}$, $\mathcal{P}%
_{k} $ at $P_{0}$ are then defined by%
\begin{equation*}
I_{\theta _{0}}^{-1}\left( \mathcal{P}\right) = V_{P_{0}}\left( \widetilde{%
\psi }\right) =E_{P_{0}}\left( \widetilde{\psi }\widetilde{\psi }^{\prime
}\right) , \qquad I_{\theta _{0}}^{-1}\left( \mathcal{P}_{k}\right)
=V_{P_{0}}\left( \widetilde{\psi }_{k}\right) ,\quad k\in \mathbb{N}^{\ast }.
\end{equation*}
From%
\begin{equation*}
\mathcal{P}_{1}\supset \mathcal{P}_{2}\supset \ldots \supset \mathcal{P}%
_{k}\supset \mathcal{P}_{k+1}\supset \ldots \supset
\bigcap\limits_{k=1}^{\infty }\mathcal{P}_{k}\supset \mathcal{P}
\end{equation*}%
we deduce that%
\begin{equation*}
\mathcal{\dot{P}}_{1}\supset \mathcal{\dot{P}}_{2}\supset \ldots \supset
\mathcal{\dot{P}}_{k}\supset \mathcal{\dot{P}}_{k+1}\supset \ldots \supset
\bigcap\limits_{k=1}^{\infty }\mathcal{\dot{P}}_{k}\supset \mathcal{\dot{P}},
\end{equation*}%
and
\begin{equation*}
\mathcal{T}_{1}\supset \mathcal{T}_{2}\supset \ldots \supset \mathcal{T}%
_{k}\supset \mathcal{T}_{k+1}\supset \ldots \supset
\bigcap\limits_{k=1}^{\infty }\mathcal{T}_{k}=\mathcal{T},
\end{equation*}%
where the last equality is due to $( \ref{spaning cond 2}) $. By Lemma 4.5
of Hansen and Sargent (1991),
\begin{equation*}
\lim_{k\rightarrow \infty }I_{\theta _{0}}^{-1}\left( \mathcal{P}_{k}\right)
=\lim_{k\rightarrow \infty }V_{P_{0}}\left( \prod\left( \overline{\psi }|%
\mathcal{T}_{k}\right) \right) =V_{P_{0}}\left( \prod\left( \overline{\psi }|%
\mathcal{T}\right) \right) =V_{P_{0}}\left( \widetilde{\psi }\right)
=I_{\theta _{0}}^{-1}\left( \mathcal{P}\right) .
\end{equation*}
\end{proof}

\subsection{Assumptions}

For a subset $A\subset \mathrm{supp}Z$, we use the following notations~: $%
g_{i,A} =g_{i}(Z,\theta_0) \mathbf{I}_{\{Z\in A\}}, $ $i=1,2,$ and
\begin{equation}  \label{b_i_A}
b_{i} = g_{i,A}-E ( g_{i,A}\ g_{j,A}^{\prime }\mid X^{( 1) },X^{( 2) }) \
E^{-1}( g_{j,A}\ g_{j,A}^{\prime }\mid X^{( 1) }, X^{( 2) }) \ g_{j,A},
\end{equation}
$( i,j) \in \{ ( 1,2), ( 2,1) \}$ where $E^{-1}\left( g_{j,A}\
g_{j,A}^{\prime }\ |\ X^{\left( 1\right) },X^{\left( 2\right) }\right) $
stands for the inverse of the matrix $E\left( g_{j,A}\ g_{j,A}^{\prime }\ |\
X^{\left( 1\right) },X^{\left( 2\right) }\right) $ that is supposed to exist.

\begin{description}
\item \label{ass_T}

\item[Assumption $\mathbf{{T}}$] \quad \newline
There exist a subset $A\subset \mathrm{supp}Z$ such that for $i=1,2,$ $%
g_{i,A}$ is a bounded function and

\begin{enumerate}
\item $E\left( g_{i,A}\ g_{i,A}^{\prime }\ |\ X^{\left( 1\right) },X^{\left(
2\right) }\right) $ is invertible and $\left\Vert E^{-1}\left( g_{i,A}\
g_{i,A}^{\prime }\ |\ X^{\left( 1\right) },X^{\left( 2\right) }\right)
\right\Vert _{\infty }<\infty $;

\item $\left\Vert E^{-1}\left( b_{i}\ b_{i}^{\prime }\ |\ X^{\left( i\right)
}\right) \right\Vert _{\infty }<\infty $ with $b_i$ defined in (\ref{b_i_A}).
\end{enumerate}
\end{description}

\begin{remark}
Under Assumption~$T$ and for any $\alpha >0$, by the definition of $b_{i}$,
for $(i,j)\in \{(1,2),\ (2,1)\},$
\begin{equation*}
E\left( g_{i}(Z,\theta _{0})\ \alpha b_{j}^{\prime }\ |\ X^{\left( 1\right)
},X^{\left( 2\right) }\right) =E\left( g_{i,A}\ \alpha b_{j}^{\prime }\ |\
X^{\left( 1\right) },X^{\left( 2\right) }\right) =0,
\end{equation*}%
and, for $i=1,2,$%
\begin{equation*}
E\left( g_{i,A}\ \alpha b_{i}^{\prime }\ |\ X^{\left( i\right) }\right)
=\alpha E\left( b_{i}\ b_{i}^{\prime }\ |\ X^{\left( i\right) }\right) .
\end{equation*}%
Therefore, in the proof of Theorem \ref{Main_Result}, up to a suitable
scaling factor, we can choose $b_{1}$ and $b_{2}$ such that conditions (\ref%
{fct_b_cond}) are satisfied.
\end{remark}

\bigskip

\begin{description}
\item[Assumption $\mathbf{{SP}}$]

\quad

\vspace{-0.35cm}

\begin{enumerate}
\item \label{ass_SP} The models $\mathcal{P}$ defined by $\left( \ref%
{model_I}\right) $ and $\mathcal{P}_{k}$ defined by $\left( \ref{model_III}%
\right) $, with $k\in \mathbb{N}^{\ast }$, can be written in the
semiparametric form
\begin{equation*}
\mathcal{P}=\left\{ P_{\theta ,\eta }:\ \theta \in \Theta ,\ \eta \in
H\right\} ,\quad \mathcal{P}_{k}=\left\{ P_{\theta ,\eta }:\ \theta \in
\Theta ,\ \eta \in H_{k}\right\} ,\ k\in \mathbb{N}^{\ast },
\end{equation*}%
and satisfy the assumptions of Lemma 25.25 (page 369) of van der
Vaart~(1998).

\item The Fisher information matrices $I_{\theta _{0}}$ and $I_{\theta
_{0}}^{\left( k\right) }$ on $\theta _{0}$ in models $\mathcal{P}$ and $%
\mathcal{P}_{k}$ respectively, for any $k\in \mathbb{N}^{\ast }$, are well
defined and nonsingular.
\end{enumerate}
\end{description}

\bigskip

To guarantee Assumption~SP.2 it suffices to suppose that for any $1\leq
j\leq J$: (i) $\|V (g_j( Z,\theta _{0})\mid X^{(j)})\|_\infty <\infty$; (ii)
the maps $\theta\mapsto E (g_j (Z,\theta _{0})\mid X^{(j)}=x^{(j)})$ are
differentiable for $P_{X^{(j)}}-$almost all $x^{(j)}$; and (iii) the
information matrix
\begin{equation*}
E\left\{E\left[ \left( \partial_{\theta ^{\prime }}g_j\left( Z,\theta
_{0}\right) \right) ^{\prime }\mid X^{(j)}\right] V^{-}\left[ g_j\left(
Z,\theta _{0}\right) \mid X^{(j)} \right] E\left[ \partial_{\theta ^{\prime
}}g_j\left( Z,\theta _{0}\mid X^{(j)} \right) \right] \right\}
\end{equation*}
is non singular.

A consequence of Assumption~SP (see Lemma 25.25 of van der Vaart~(1998)) is
that the parameter defined by $\psi \left( P_{\theta ,\eta }\right) =\theta $
is differentiable at $P_{0}=P_{\theta _{0},\eta _{0}}$ with respect to the
tangent space $\mathcal{T}=\mathcal{T}\left( \mathcal{P},P_{0}\right) $. It
also ensures that the tangent space $\mathcal{T}$ can be written as the sum
of the finite dimensional subspace spanned by the components of the
parametric score $S_{\theta _{0}}$ and the tangent space $\mathcal{T}%
^{\prime }$ corresponding to the nonparametric part $\mathcal{P}^{\prime
}=\left\{ P_{\theta _{0},\eta }:\ \eta \in H\right\} $ of the model $%
\mathcal{P}$~:%
\begin{equation*}
\mathcal{T}=\mathrm{lin}S_{\theta _{0}}+\mathcal{T}^{\prime }.
\end{equation*}%
Note that this assumption does not necessarily mean that the parameters $%
\theta $ and $\eta $ are completely separated. In fact $\theta$ and $\eta$
are connected since the functional parameter $\eta $ can have $\theta $
among its arguments. Assumption~$SP$ only means that when considering the
density of $P_{\theta ,\eta }$ with respect to a dominating measure $\mu $
we could write it under the form
\begin{equation*}
f\left( \cdot ,\theta ,\eta \left( v\left( \cdot ,\theta \right) \right)
\right) ,
\end{equation*}%
with $f$ and $v$ having a known form, where $f\left( \cdot ,\theta _{0},\eta
\left( v\left( \cdot ,\theta _{0}\right) \right) \right) $ and $f\left(
\cdot ,\theta ,\eta _{0}\left( v\left( \cdot ,\theta \right) \right) \right)
$ belong to the model $\mathcal{P}$ for every $\theta \in \Theta $ and $\eta
\in H$. For example, in the conditional mean setting with one conditioning
vector
\begin{equation*}
E\left[ Y-m\left( X,\theta \right) \ |\ X\right] =0,
\end{equation*}%
we can take $H$ as the set of zero conditional mean densities of $%
Z=(Y^\prime,X^\prime)^\prime$, i.e.
\begin{eqnarray*}
H &=&\left\{ p\left( y,x\right) \cdot \gamma \left( x\right) :\ p\geq 0,\
\gamma \geq 0,\ \int p\left( y,x\right) dy=1,\ \int yp\left( y,x\right)
dy=0,\ \forall x,\right. \\
&& \\
&&\qquad \left. \int \gamma \left( x\right) dy=1\right\}
\end{eqnarray*}%
and $v\left( y,x,\theta \right) =\left( y-m\left( x,\theta \right) ,x\right)
$, so that
\begin{equation*}
\eta \left( v\left( z,\theta \right) \right) =\eta \left( y-m\left( x,\theta
\right) ,x\right) =p\left( y-m\left( x,\theta \right) ,x\right) \cdot \gamma
\left( x\right)
\end{equation*}
and
\begin{equation*}
f\left( z,\theta ,\eta \left( v\left( z,\theta \right) \right) \right) =\eta
\left( v\left( z,\theta \right) \right) .
\end{equation*}
In the proof of Theorem~\ref{Main_Result} we identify the density $%
f(\cdot,\theta,\eta(v(\cdot,\theta)))$ with the infinite dimensional
nuisance parameter $\eta $ which is itself a density.

\subsection{Contraction property in regression-like models with missing data}\label{rrz_contrac}

With the same notation of  subsection \ref{sec_appli_missing}, we shall prove that the equation
\begin{eqnarray}
a_{1}^{\ast }\left( X^{\ast }\right) &=&E\left\{ E\left[ a_{1}^{\ast }\left(
X^{\ast }\right) \ \rho \left( Z,\theta _{0}\right) \ |\ W\right] \ \frac{%
1-\pi \left( W\right) }{\pi \left( W\right) }\ \rho ^{\prime }\left(
Z,\theta _{0}\right) \ |\ X^{\ast }\right\}  \notag \\
&&  \label{contraction 1} \\
&&\times \ E^{-1}\left[ \frac{1}{\pi \left( W\right) }\ \rho \left( Z,\theta
_{0}\right) \ \rho ^{\prime }\left( Z,\theta _{0}\right) \ |\ X^{\ast }%
\right]  \notag
\end{eqnarray}%
has a unique solution which can be obtained by successive approximation, under
the additional assumption%
\begin{equation}
\inf_{w}\pi \left( w\right) =1-\beta >0,  \label{contraction 2}
\end{equation}%
the infimum being taken over all possible values of $W.$ For simplicity, in the reminder of this subsection we drop the arguments of the functions. Let $\widetilde{\rho }=\pi^{-1/2} \rho $. Assuming that $%
E\left( \widetilde{\rho }\ \widetilde{\rho }^{\ \prime }\ |\ X^{\ast }\right) $
is invertible, equation $\left( \ref{contraction 1}\right) $ can be
equivalently written under the form%
\begin{eqnarray*}
a_{1}^{\ast }\ \widetilde{\rho } &=&E\left[ E\left( a_{1}^{\ast }\ \rho \ |\
W\right) \ \frac{1-\pi }{\pi }\ \rho ^{\prime }\ |\ X^{\ast }\right] \
E^{-1}\left( \frac{1}{\pi }\ \rho \ \rho ^{\prime }\ |\ X^{\ast }\right) \
\widetilde{\rho } \\
&=&E\left[ E\left( a_{1}^{\ast }\ \widetilde{\rho }\ |\ W\right) \ \left(
1-\pi \right) \ \widetilde{\rho }^{\ \prime }\ |\ X^{\ast }\right] \
E^{-1}\left( \widetilde{\rho }\ \widetilde{\rho }^{\ \prime }\ |\ X^{\ast
}\right) \ \widetilde{\rho } \\
&=:&\widetilde{T}\left(
a_{1}^{\ast }\ \widetilde{\rho }\right) .
\end{eqnarray*}%
We will show that the map $\widetilde{T}$ is a contraction.
Before that, let us state a Cauchy-Schwarz inequality for matrix valued random variables, a version of an inequality in Lavergne (2008):
let $\mathbb{E}$ denote the conditional expectation given an arbitrary $\sigma-$field, let $A\in\mathbb{R}^n \times \mathbb{R}^p$ and  $B \in \mathbb{R}^n \times \mathbb{R}^q$ be random matrices such that
$\mathbb{E}(tr(A^\prime A)),\mathbb{E}(tr(B^\prime B))<\infty$ and $\mathbb{E}(A^\prime A)$ is non-singular.
Then $\mathbb{E}(B^\prime B) - \mathbb{E}(B^\prime A) \mathbb{E}^{-1}(A^\prime A) \mathbb{E}(A^\prime B)$ is positive semi-definite, with equality iff  $B = A \mathbb{E}^{-1}(A^\prime A)\mathbb{E}(A^\prime B)$.\footnote{Like in Lavergne (2008), let $\Lambda =\mathbb{E}^{-1}(A^\prime A) \mathbb{E}(A^\prime B)$. Then
$$
\mathbb{E} [(B-A\Lambda)^\prime (B-A\Lambda)=\mathbb{E}(B^\prime B) - \mathbb{E}(B^\prime A) \mathbb{E}^{-1}(A^\prime A) \mathbb{E}(A^\prime B)
$$
is clearly positive semi-definite,  and is zero iff $B = A\Lambda.$}
We also use the following notation: for any symmetric matrices $B_1,B_2$,  $B_1 \gg B_2$ means $B_1 - B_2$ is positive semi-definite.
Let us write
\begin{eqnarray*}
E[ \widetilde{T}\left( a_{1}^{\ast }\ \widetilde{\rho }\right) \
\widetilde{T}^{\prime }\left( a_{1}^{\ast }\ \widetilde{\rho }\right)]
&=&E\left\{ \left[ E\left( a_{1}^{\ast }\ \widetilde{\rho }\ |\ W\right) \
\left( 1-\pi \right) \ \widetilde{\rho }^{\ \prime }\ |\ X^{\ast }\right] \
E^{-1}\left( \widetilde{\rho }\ \widetilde{\rho }^{\ \prime }\ |\ X^{\ast
}\right) \ \widetilde{\rho }\right. \\
&&\times \left. \widetilde{\rho }^{\ \prime }\ E^{-1}\left( \widetilde{\rho }\
\widetilde{\rho }^{\ \prime }\ |\ X^{\ast }\right) \ \left\{ \left[ E\left(
a_{1}^{\ast }\ \widetilde{\rho }\ |\ W\right) \ \left( 1-\pi \right) \
\widetilde{\rho }^{\ \prime }\ |\ X^{\ast }\right] \right\} ^{\prime }\
\right\} \\
&=&E\left\{ \left[ E\left( a_{1}^{\ast }\ \widetilde{\rho }\ |\ W\right) \
\left( 1-\pi \right) \ \widetilde{\rho }^{\ \prime }\ |\ X^{\ast }\right] \
E^{-1}\left( \widetilde{\rho }\ \widetilde{\rho }^{\ \prime }\ |\ X^{\ast
}\right) \ \right. \\
&&\times \left. \left\{ \left[ E\left( a_{1}^{\ast }\ \widetilde{\rho }\ |\
W\right) \ \left( 1-\pi \right) \ \widetilde{\rho }^{\ \prime }\ |\ X^{\ast }%
\right] \right\} ^{\prime }\ \right\} \\
(\text{Cauchy-Schwarz})\quad &\ll &E\left\{ E\left[ E\left( a_{1}^{\ast }\ \widetilde{\rho }%
\ |\ W\right) \ \left( 1-\pi \right) ^{2}\ E\left( \widetilde{\rho }^{\ \prime
}\ a_{1}^{\ast \prime }\ |\ W\right) \ |\ X^{\ast }\right] \ \right\} \\
&=&E\left[ E\left( a_{1}^{\ast }\ \widetilde{\rho }\ |\ W\right) \ \left(
1-\pi \right) ^{2}\ E\left( \widetilde{\rho }^{\ \prime }\ a_{1}^{\ast \prime
}\ |\ W\right) \right] \\
(\text{Cauchy-Schwarz})\quad &\ll &E\left[ \left( 1-\pi \right) ^{2}\ \left( a_{1}^{\ast }\
\widetilde{\rho }\right) \ \left( a_{1}^{\ast }\ \widetilde{\rho }\right)
^{\prime }\right]
\end{eqnarray*}
This implies
\begin{eqnarray*}
\left\Vert \widetilde{T}\left( a_{1}^{\ast }\ \widetilde{\rho }\right)
\right\Vert _{L^{2}}^{2} &=&E\left\{ tr\left[ \widetilde{T}^{\prime }\left(
a_{1}^{\ast }\ \widetilde{\rho }\right) \ \widetilde{T}\left( a_{1}^{\ast }\
\widetilde{\rho }\right) \right] \right\} =tr\left\{ E\left[ \widetilde{T}%
\left( a_{1}^{\ast }\ \widetilde{\rho }\right) \ \widetilde{T}^{\prime
}\left( a_{1}^{\ast }\ \widetilde{\rho }\right) \right] \right\} \\
&& \\
&\leq &\sup_{w}\left[ 1-\pi \left( w\right) \right] \ \left\Vert a_{1}^{\ast
}\ \widetilde{\rho }\right\Vert _{L^{2}}^{2}\leq \beta \ \left\Vert
a_{1}^{\ast }\ \widetilde{\rho }\right\Vert _{L^{2}}^{2},
\end{eqnarray*}%
where $\beta =\sup\limits_{w}\left[ 1-\pi \left( w\right) \right]
=1-\inf\limits_{w}\pi \left( w\right) <1$ by assumption $\left( \ref%
{contraction 2}\right) $. Deduce that $\widetilde T$ is a contracting map.

\subsection{Efficient score with parametric selection probability in \\regression-like models with missing data}\label{rrz_selec_prob}
Let $X^{\left( 1\right) }=X^{\ast }$, $ X^{\left(
2\right) }=W$ and the parameter vector $\theta = (\alpha^\prime, \gamma^\prime)^\prime$. Moreover, let
\begin{eqnarray*}
g_{1} (Z,\theta) &=&\dfrac{\delta }{\pi \left( W,\gamma \right) }\ \rho \left(
Y,X^{\ast },\alpha \right) ,\qquad
g_{2} (Z,\theta)=\dfrac{\delta }{\pi \left( W,\gamma \right) }-1, \\
\overline{S}_{\theta } &=&\overline{a}_{1}\left( X^{\ast }\right) \ g_{1}(Z,\theta)+%
\overline{a}_{2}\left( W\right) \ g_{2}(Z,\theta)=\left(
\begin{array}{c}
\overline{S}_{\alpha } \\
\overline{S}_{\gamma }%
\end{array}%
\right) ,
\end{eqnarray*}%
where%
\begin{eqnarray*}
\overline{a}_{1}\left( X^{\ast }\right) &=&\overline{a}_{1}\left( X^{\left(
1\right) }\right) \\
&=&\left(
\begin{array}{c}
-E\left[ E\left( \pi^{-1}( W,\gamma_0)\delta \mid X^{\ast },W\right)  \partial
_{\alpha }\rho ^{\prime }\mid X^{\ast }\right]  E^{-1}\left( \pi^{-1}( W,\gamma_0) \rho \ \rho ^{\prime }
 \mid X^{\ast }\right) \\
\\
0%
\end{array}%
\right) \\
&+&E\left\{ E\left[ \overline{a}_{1}\left( X^{\ast }\right)  \rho
|\ W\right] \ \left( \pi^{-1} ( W,\gamma_0)-1\right) \rho ^{\prime }\mid  X^{\ast
}\right\} E^{-1}\left(
\pi^{-1} ( W,\gamma_0) \rho  \rho ^{\prime }\ |\ X^{\ast }\right) .
\end{eqnarray*}%
If we partition $\overline{a}_{1}\left( X^{\ast }\right) $ in $\overline{a}%
_{1}\left( X^{\ast }\right) =\left(
\begin{array}{c}
\overline{a}_{1,\alpha }\left( X^{\ast }\right) \\
\overline{a}_{1,\gamma }\left( X^{\ast }\right)%
\end{array}%
\right) $ and we use the same short notation as previously,  the preceding equations can be written as%
\begin{eqnarray*}
\overline{a}_{1,\alpha }\left( X^{\ast }\right) &=&-E\left( E\left( \dfrac{%
\delta }{\pi }\ |\ X^{\ast },W\right) \ \partial _{\alpha }\rho ^{\prime }\
|\ X^{\ast }\right) \ E^{-1}\left( \dfrac{1}{\pi }\ \rho \ \rho ^{\prime }\
|\ X^{\ast }\right) \\
&&\qquad +E\left\{ E\left[ \overline{a}_{1,\alpha }\left( X^{\ast }\right) \
\rho \ |\ W\right] \ \left( \dfrac{1}{\pi }-1\right) \ \rho ^{\prime }\ |\
X^{\ast }\right\} \\
&&\qquad \qquad \qquad \qquad \qquad \qquad \qquad \times \ E^{-1}\left(
\dfrac{1}{\pi }\ \rho \ \rho ^{\prime }\ |\ X^{\ast }\right) , \\
\overline{a}_{1,\gamma }\left( X^{\ast }\right) &=&E\left\{ E\left[
\overline{a}_{1,\gamma }\left( X^{\ast }\right) \ \rho \ |\ W\right] \
\left( \dfrac{1}{\pi }-1\right) \ \rho ^{\prime }\ |\ X^{\ast }\right\} \\
&&\qquad \qquad \qquad \qquad \qquad \qquad \qquad \times \ E^{-1}\left(
\dfrac{1}{\pi }\ \rho \ \rho ^{\prime }\ |\ X^{\ast }\right) ,
\end{eqnarray*}%
with the obvious solution $\overline{a}_{1,\gamma }\equiv 0$ for the
subvector of $\overline{a}_{1}$ corresponding to $\gamma $ (possibly not the
unique solution, but any solution yields the same efficient score $\overline{%
S}_{\theta }$). Similar calculations can be done for $\overline{a}_{2}\left(
W\right) ~:$

\begin{eqnarray*}
\overline{a}_{2}\left( W\right) &=&\overline{a}_{2}\left( X^{\left( 2\right)
}\right) \\
&& \\
&=&\left(
\begin{array}{c}
0 \\
\\
\dfrac{1}{\pi }\ \partial _{\gamma }\pi%
\end{array}%
\right) \ \dfrac{\pi }{1-\pi }-E\left[ \overline{a}_{1}\left( X^{\ast
}\right) \ \rho \ |\ W\right] , \\
&&
\end{eqnarray*}%
which gives, for $\overline{a}_{2}\left( W\right) =\left(
\begin{array}{c}
\overline{a}_{2,\alpha }\left( W\right) \\
\overline{a}_{2,\gamma }\left( W\right)%
\end{array}%
\right) $,%
\begin{eqnarray*}
\overline{a}_{2,\alpha }\left( W\right) &=&-E\left[ \overline{a}_{1,\alpha
}\left( X^{\ast }\right) \ \rho \ |\ W\right] \\
&& \\
\overline{a}_{2,\gamma }\left( W\right) &=&\dfrac{1}{1-\pi }\ \partial
_{\gamma }\pi -E\left[ \overline{a}_{1,\gamma }\left( X^{\ast }\right) \
\rho \ |\ W\right] =\dfrac{1}{1-\pi }\ \partial _{\gamma }\pi .
\end{eqnarray*}%
Therefore,%
\begin{equation*}
\overline{S}_{\theta }=\overline{a}_{1}\left( X^{\ast }\right) \ g_{1}+%
\overline{a}_{2}\left( W\right) \ g_{2}=\left(
\begin{array}{c}
\overline{S}_{\alpha } \\
\\
\overline{S}_{\gamma }%
\end{array}%
\right) =\left(
\begin{array}{c}
\overline{a}_{1,\alpha }\left( X^{\ast }\right) \ g_{1}+\overline{a}%
_{2,\alpha }\left( W\right) \ g_{2} \\
\\
\overline{a}_{2,\gamma }\left( W\right) \ g_{2}%
\end{array}%
\right) ,
\end{equation*}%
where%
\begin{eqnarray*}
\overline{a}_{1,\alpha }\left( X^{\ast }\right) &=&-E\left( \partial
_{\alpha }\rho ^{\prime }\ |\ X^{\ast }\right) \ E^{-1}\left( \dfrac{1}{\pi }%
\ \rho \ \rho ^{\prime }\ |\ X^{\ast }\right) \\
&&\qquad +E\left\{ E\left[ \overline{a}_{1,\alpha }\left( X^{\ast }\right) \
\rho \ |\ W\right] \ \dfrac{1-\pi }{\pi }\ \rho ^{\prime }\ |\ X^{\ast
}\right\} \\
&&\qquad \qquad \qquad \qquad \qquad \qquad \qquad \times \ E^{-1}\left(
\dfrac{1}{\pi }\ \rho \ \rho ^{\prime }\ |\ X^{\ast }\right) , \\
&& \\
\overline{a}_{2,\alpha }\left( W\right) &=&-E\left[ \overline{a}_{1,\alpha
}\left( X^{\ast }\right) \ \rho \ |\ W\right] , \\
&& \\
\overline{a}_{2,\gamma }\left( W\right) &=&\dfrac{\pi }{1-\pi }\ \partial
_{\gamma }\pi \qquad \left( =\dfrac{\pi \left( W,\gamma _{0}\right) }{1-\pi
\left( W,\gamma _{0}\right) }\ \partial _{\gamma }\pi \left( W,\gamma
_{0}\right) \right) .
\end{eqnarray*}%
Now, for any $s=b\left( W\right) \cdot g_{2}=b\left( W\right) \left( \dfrac{%
\delta }{\pi \left( W,\gamma _{0}\right) }-1\right) \in \mathcal{T}%
_{2}^{\perp }$, we have%
\begin{eqnarray*}
E\left( \overline{S}_{\alpha }\ s^{\prime }\ |\ W\right) &=&E\left[
\overline{S}_{\alpha }\ \left( \dfrac{\delta }{\pi \left( W,\gamma
_{0}\right) }-1\right) \ |\ W\right] \ b^{\prime }\left( W\right) \\
&& \\
&=&E\left\{ \left[ \overline{a}_{1,\alpha }\left( X^{\ast }\right) \ \dfrac{%
\delta }{\pi }\ \rho +\overline{a}_{2,\alpha }\left( W\right) \ \left(
\dfrac{\delta }{\pi }-1\right) \right] \ \left( \dfrac{\delta }{\pi }%
-1\right) \ |\ W\right\} \ b^{\prime }\left( W\right) \\
&& \\
&=&\left\{ E\left[ \overline{a}_{1,\alpha }\left( X^{\ast }\right) \ \rho \
|\ W\right] +\overline{a}_{2,\alpha }\left( W\right) \right\} \ \left(
\dfrac{1}{\pi }-1\right) \ b^{\prime }\left( W\right) \\
&& \\
&=&\left\{ E\left[ \overline{a}_{1,\alpha }\left( X^{\ast }\right) \ \rho \
|\ W\right] -E\left[ \overline{a}_{1,\alpha }\left( X^{\ast }\right) \ \rho
\ |\ W\right] \right\} \ \left( \dfrac{1}{\pi }-1\right) \ b^{\prime }\left(
W\right) \\
&& \\
&=&0,
\end{eqnarray*}%
so that, since $\overline{S}_{\gamma }=\overline{a}_{2,\gamma }\left(
W\right) \cdot g_{2}$, we obtain%
\begin{equation*}
E\left( \overline{S}_{\alpha }\ \overline{S}_{\gamma }^{\prime }\right) =E%
\left[ E\left( \overline{S}_{\alpha }\ \overline{S}_{\gamma }^{\prime }\ |\
W\right) \right] =0.
\end{equation*}%
This means that the efficient score $S_{\alpha }^{\ast }$ for $\alpha $,
equal to the residual of the (componentwise) projection of $\overline{S}%
_{\alpha }$ on $\overline{S}_{\gamma }$, coincides with $\overline{S}%
_{\alpha },$%
\begin{equation*}
S_{\alpha }^{\ast }=\overline{S}_{\alpha }-E\left( \overline{S}_{\alpha }\
\overline{S}_{\gamma }^{\prime }\right) \ V^{-1}\left( \overline{S}_{\gamma
}\right) \ \overline{S}_{\gamma }=\overline{S}_{\alpha },
\end{equation*}%
and has the same expression, as already noticed in Robins, Rotnitzky and
Zhao~(1994), as in the case where $\pi \left( W\right) $ is completely
known~:%
\begin{eqnarray*}
S_{\alpha }^{\ast } &=&\overline{S}_{\alpha }=\overline{a}_{1,\alpha }\left(
X^{\ast }\right) \ g_{1}+\overline{a}_{2,\alpha }\left( W\right) \ g_{2} \\
&& \\
&=&a_{1}^{\ast }\left( X^{\ast }\right) \ g_{1}+a_{2}^{\ast }\left( W\right)
\ g_{2}.
\end{eqnarray*}

\end{document}